\documentclass{amsart}

\usepackage{amscd,amssymb,epsfig, epsf,pinlabel,graphicx}
\usepackage{epic,eepic}
\usepackage[dvipsnames]{xcolor}

\usepackage[all]{xy}
\SelectTips{cm}{}
\allowdisplaybreaks

\numberwithin{equation}{subsection}

\newtheorem{propo}{Proposition}[section]
\newtheorem{corol}[propo]{Corollary}
\newtheorem{theor}[propo]{Theorem}
\newtheorem{lemma}[propo]{Lemma}

\theoremstyle{definition}

\theoremstyle{remark}

\newcommand{\egm}{\normalcolor{}}

\let\oldmarginpar\marginpar
\renewcommand\marginpar[1]{\oldmarginpar{\footnotesize #1}}

\newcommand{\CC}{\mathbb{C}}
\newcommand{\QQ}{\mathbb{Q}}
\newcommand{\ZZ}{\mathbb{Z}}

\newcommand{\RR}{\mathbb{R}}

\newcommand{\Hom}{\operatorname{Hom}}
\newcommand{\Ker}{\operatorname{Ker}}

\newcommand{\Int}{\operatorname{Int}}

\newcommand{\card}{\operatorname{card}}

\newcommand{\id}{\operatorname{id}}
\newcommand{\aug}{\operatorname{aug}}

\begin{document}

\title[Topological constructions of tensor fields on moduli spaces]{Topological constructions of tensor fields on moduli spaces}

    \author[Vladimir Turaev]{Vladimir Turaev}
    \address{
    Vladimir Turaev \newline
    \indent   Department of Mathematics \newline
    \indent  Indiana University \newline
    \indent Bloomington IN47405, USA\newline
    \indent $\mathtt{vturaev@yahoo.com}$}

\begin{abstract} We show how    topology   of a space  may lead to tensor fields   on (the smooth part of)  moduli spaces of the    fundamental group.   %Our  key examples    involve so-called gates and   quasi-surfaces. 
\end{abstract}

% For  the latter,  our brackets generalize    Goldman's Lie brackets  of   surfaces.  The algebraic aspects of these constructions   are formulated in terms of trace algebras. 

\maketitle

\section {Introduction}

%Loops in   topological spaces have been extensively  used   since the foundational work of Henri Poincar\'e on 

Moduli spaces of the  fundamental groups  of  surfaces    carry beautiful geometric structures,   in particular,  Poisson brackets, see  \cite{AB}, \cite{Wo}, \cite{FR}.    These brackets were described by  Goldman 
\cite{Go1}, \cite{Go2}   in terms of  a  Lie bracket in the module   of loops in the surface, see 
  \cite{AKKN1}, \cite{AKKN2},   \cite{Ka}, \cite{LS} for recent work on Goldman's bracket.
   Here we     extend this line of study.  Our starting point is  the work of    van den Bergh \cite{VdB} and   Crawley-Boevey \cite{Cb} who  derive from  any algebra~$A$ and an integer $n\geq 1$  the (commutative) coordinate  algebra  $A_n$   of  the affine scheme $ {\text {Rep}}_n(A)$ of $n$-dimensional representations of~$A$.  These authors  also define a subalgebra $A^t_n \subset A_n$   which -  under appropriate   assumptions   -     is the coordinate  algebra of  the affine quotient scheme ${\text {Rep}}_n(A)//GL_n$. We  view the latter affine scheme  as  the  moduli space  of $n$-dimensional representations of~$A$. Inspired by the interpretation of vector fields on a smooth manifold  as derivations of the algebra of smooth functions on this manifold, we can define vector fields  on ${\text {Rep}}_n(A)//GL_n$  as  derivations of the algebra $A^t_n$. More generally, for any integers $m, n\geq 1$, we can define $m$-tensor fields on ${\text {Rep}}_n(A)//GL_n$  as $m$-linear forms  
$(A^t_n)^m  \to A^t_n$ (or $(A^t_n)^{\otimes m}  \to A^t_n$) which are derivations in all~$m$  variables. Despite a purely algebraic formulation,  this  approach   may lead to smooth tensor fields on the smooth parts of moduli  spaces, see  \cite{MTnew}.   To construct  $m$-linear  forms       in $A^t_n$ which are derivations in all variables, we use a method inspired by the work of    Crawley-Boevey \cite{Cb} on Poisson structures. Namely, we set 
  $\check A=A/[A,A]$ and      derive such  $m$-linear  forms       in $A^t_n$   from   $m$-linear   forms  
$\check A^m  \to \check A$  satisfying certain  assumptions. We  call   $m$-linear  forms  in $ \check A$  satisfying these  assumptions 
   \emph{$m$-braces in~$A$}.

% We    show how to  construct   braces  in  group algebras starting  from   Fox derivatives.  

%For example, a Poisson  bracket  on  $M_n(A)$ is just a Poisson bracket in  $A^t_n$. A vector field on $M_n(A)$  is an algebra derivation $A^t_n \to A^t_n$ (and similarly for tensor fields). By \cite{Cb},  useful  operations in $A^t_n$ can be derived from operations in the zero  Hochschild homology $\check A=A/[A,A]$   of~$A$.  

Our  main   aim is a construction of   braces in the group algebras of the fundamental groups of   topological  spaces.  We give two such  constructions. First, consider a   topological space~$X$ and let~$A$ be  the group algebra  of $\pi_1(X)$. 
 % Then~$\check A $ is the free module generated by   free homotopy classes of loops in~$X$, and   braces in~$A$ may be defined in terms of   geometric operations on   loops. 
A \emph{gate} in~$X$ is a   path-connected subspace $C\subset X$ such  that all loops in~$C$ are contractible in~$X$ and~$C$ has a   cylinder  neighborhood $C\times [-1,1]$ in~$X$.  We show that    a gate   in~$X$ gives rise to   an $m$-brace  in~$A$ for all $m\geq 1$. This   \lq\lq gate  brace''    induces     $m$-tensor fields on the moduli spaces of $\pi_1(X)$  for  all  $m\geq 1$.    For example, if~$X$ is a surface with boundary, then any properly embedded segment  in~$X$ is a gate; so, it determines an $m$-brace in~$A$   and an   $m$-tensor field  on the moduli space ${\text {Rep}}_n(A)//GL_n$    for  all  $m, n\geq 1$.

Our  second construction of braces  applies to  so-called  quasi-surfaces which we introduce here as  generalizations of  the usual   surfaces with boundary. A  quasi-surface, $X$,  is  obtained by   gluing a  surface~$\Sigma$    to an arbitrary      topological space      along  a finite set of  disjoint segments in $\partial \Sigma$. These  segments  give rise to gates in~$X$ which   split~$X$   into  the  surface  part (a copy of~$\Sigma$) and the  singular   part (the rest).   By the above, each   gate induces an $m$-brace  in the group algebra $A$ of $\pi_1(X)$ for all $m\geq 1$. For oriented~$\Sigma$, we   use intersections  of loops to define a skew-symmetric \lq\lq intersection  2-brace''  in~$A$  generalizing  the Goldman  bracket of surfaces. Our main result is a  Jacobi-type identity relating the intersection 2-brace  to the gate 3-braces.  This   generalizes  to quasi-surfaces the Jacobi identity for the Goldman  bracket of surfaces. The intersection 2-brace induces   skew-symmetric bilinear forms $\{A^t_n \times A^t_n \to  A^t_n\}_{n\geq 1}$, and our Jacobi-type identity relates them  to the trilinear forms in $\{A^t_N\}_n$ derived from the gate 3-braces.
 For completeness, we also define    intersection pairings in  1-homology of  quasi-surfaces generalizing  the usual  intersection pairings in  1-homology of   surfaces.

Any   surface~$\Sigma$ with boundary   may be viewed as a quasi-surface in multiple ways determined by a  choice of  disjoint properly embedded segments in~$\Sigma$    splitting~$ \Sigma$ into   the  \lq\lq surface part'' and the  \lq\lq singular part''.  For oriented~$\Sigma$, each such splitting determines   a 2-brace  in the group algebra $A$ of $\pi_1(\Sigma)$ and the induced pairings $\{A^t_n \times A^t_n \to  A^t_n\}_{n\geq 1}$. By the above, these brace  and pairings  satisfy  Jacobi-type identities involving    the   3-braces associated with the segments  in question. 

%A similar identity holds for the tensor  fields determined by these braces. 
%  The study of surfaces is known to  lead  to a  Lie cobracket  complementing the Goldman bracket, see~%\cite{Tu2}.
%The theory exposed here has  similar ramifications  which will be discussed elsewhere.

 %  (for   recent work, see  \cite{AKKN1}, \cite{AKKN2}, \cite{Hain}).
 
 %This gives,  in particular, a family of 
 %Poisson brackets on the moduli spaces of $\Sigma_0$ numerated by such choices. 
 
  %Our  geometric methods are  based on     cutting and pasting   loops    in the     surface part of a quasi-surface and   near the gates. The pieces of   loops lying in the singular part play a  passive role, i.e.,    are kept without   modifications. 
%This   turns out to be  sufficient for our aims.
%On the algebraic side we use  the work of W. Crawley-Boevey \cite{Cb}   on  $H_0$-Poisson   structures on algebras and the work of G. Massuyeau and the author on Fox pairings, see \cite{MTold}.           
%As a consequence, each (oriented) surface~$  \Sigma$ embedded  in a topological space  so that it meets the rest of this space only   along  several segments in $\partial \Sigma$ produces    Lie brackets in   abelian groups and commutative algebras  naturally associated with the ambient space.  
 
%  In the sequel to this paper, the author plans to consider generalized quasi-surfaces obtained by gluing surfaces to arbitrary topological spaces   along segments in the boundary and circle components of the boundary. 
%
%The author would like to thank  A. Ramadoss for a helpful dicussion.

The first part of the paper (Sections \ref{AMdddfT0}--\ref{Quasi-Lie brackets and brace algebras}) presents our  algebraic methods and the second part  (Sections \ref{A topological example}--\ref{section3ghgh}) is devoted to   topological constructions.

This work was  supported by the NSF grant DMS-1664358.

%  called \emph{coordinate algebras}.

%\part*{\quad \quad \,\, \quad \quad \quad \quad \quad   \quad \quad \quad \quad \quad Part I. Braces}

%\vspace{2mm}

\section{Preliminries}\label{coordinate2}\label{AMdddfT0}

We   briefly   recall  representation schemes  and trace   algebras following \cite{VdB}, 
   \cite{Cb}. Then we   discuss derivations in   algebras.

 \subsection{Representation schemes}\label{coordinateddd2w-}  Throughout the paper we fix a commutative base ring~$R$. By a module we mean an $R$-module and by an algebra we mean (unless explicitly stated to the contrary)  an associative $R$-algebra  with unit.
We associate with every  algebra~$A$  and an integer $n\geq 1$ an affine scheme  ${\rm {Rep}}_n (A)$, the  \emph{$n$-th  representation scheme} of~$A$.   For any commutative algebra~$S$,  the set of $S$-valued points of ${\rm {Rep}}_n (A)$  is the set of
 algebra homomorphisms  $A \to {\rm Mat}_n(S)$. The coordinate ring, $A_n$,  of ${\rm {Rep}}_n (A)$
is    generated (over~$R$)  by
the symbols $x_{ij}$ with $x\in A$ and $i,j  \in  \{1,2, \ldots, n\}$. These generators commute  
and satisfy the following relations:     $1_{ij}=\delta_{ij}$ for all $i,j$,  where
$\delta_{ij}$ is the Kronecker delta; for all $x,y\in A$, $ r\in R $,  and $i,j\in \{1,2, \ldots, n\}$,
$$  (rx)_{ij}= r x_{ij}, \quad    (x + y)_{ij}=x_{ij}+ y_{ij} \quad {\rm {and}} \quad (xy)_{ij}= \sum_{l=1}^n \,  x_{il}\,
y_{lj}. $$
The function on the set of $S$-valued points of ${\rm {Rep}}_n (A)$  determined by    $x_{ij}$ assigns to a homomorphism $f: A \to {\rm Mat}_n(S) $   the $(i,j)$-entry of the  matrix $f(x)  $.   That these functions satisfy the relations   above  is straightforward. 

The  action of the group $G=GL_n(R)$  on 
$ \Hom (A,  {\rm Mat}_n(S))$ by conjugations  induces an action of~$G$ on  the commutative algebra $A_n$ for all~$n$. Explicitly, for $g=(g_{kl})\in G$ and any $x\in A, i,j \in \{1,..., n\}$ we have
$$g \cdot x_{ij}=\sum_{k=1}^n \sum_{l=1}^n g_{ik} (g^{-1})_{lj} x_{kl}.$$
The set of invariant elements $A_n^G=\{a\in A_n\, \vert \, Ga =a \}   $ is a subalgebra of $A_n$. 
This   is     the coordinate algebra of the affine quotient scheme ${\rm Rep}_n (A)//G$  which we view as the \lq \lq moduli  space'' of $n$-dimensional representations of~$A$.

\subsection{The module $\check A$ and the  trace}\label{coordinateddd2w} Given an  algebra~$A$,
 let  $A'=[A,A]$ be the submodule of~$A$ spanned by the commutators $ xy-yx $ with $ x,y \in A$.   The quotient module $\check A=A/A'$ is  the zeroth Hochschild homology of~$A$.  Now, for any integer $n\geq 1$, the linear map $A\to A_n, x \mapsto
\sum_{i=1}^n  x_{ii} $   is      called the \emph{trace} and denoted    ${\rm tr}$.  The trace  
 annighilates all the commutators in~$A$ and therefore  ${\rm tr}(A')=0$. 
Thus, the trace induces a
  linear map $ \check A \to A_n$     also denoted~${\rm tr}$.

%This definition is motivated by a study of  the algebra of $R$-valued functions on   the set of   algebra homomorphisms $ \Hom (A,  {\rm Mat}_n(R))$.  There is a canonical algebra homomorphism  from $A_n$ to   this algebra of  functions; it  carries each generator $x_{ij}\in  A_n$  to the   function   whose value on a   homomorphism $f: A \to {\rm Mat}_n(R) $ is the $(i,j)$-entry of the  matrix $f(x)  $. 

%We   define two  subalgebras of $A_n$. Observe first  that the Lie algebra of $n\times n$-matrices $\mathfrak g={\rm Mat}_n(R)$ with Lie bracket $[g,h]=gh-hg$    %acts on $A_n$ by $g x_{ij}=\sum_{k=1}^n ( g_{kj} x_{ik}-g_{ik} x_{kj})$ for $g  \in \mathfrak g$ and all $i,j$.  The $\mathfrak g$-invariant elements of $A_n$ %form a subalgebra $A_n^{\mathfrak g}$ of $A_n$.

  The subalgebra of $A_n$ generated by ${\rm {tr}} (A) ={\rm {tr}} (\check A)  $ is  called the \emph{$n$-th trace algebra of~$A$}  and is  denoted $A^t_n$.  A direct computation shows that  ${\rm {tr}} (A) \subset A_n^G $ and therefore $A^t_n \subset A_n^G $.   If  the ground ring~$R$ is an algebraically closed field of  characteristic zero and~$A$ is a finitely generated algebra,  then  a theorem of   Le Bruyn and Procesi \cite{Pro}   implies that    $A^t_n=A^G_n$ so that  $A^t_n$ is         the coordinate algebra of     ${\rm Rep}_n (A)//G$.
    
 \subsection{Derivations} \label{coordrerdd2w}   A {\it derivation} of an algebra~$  {A}$
 is a linear map $d:{   A}\to {   A}$ such that
   $d(xy)= d(x) y+  x d(y)$    for all $x,y \in  {   A}$. We denote by ${\rm {Der}} ( A)$  the module of    derivations of~$ A$. 
   Given   $d_1, d_2 \in {\rm {Der}} (  A)$, the commutator $[d_1, d_2]=d_1 \circ d_2 -d_2 \circ d_1$ is  a derivation  of~$ {A}$. This defines a Lie bracket $[-,-]$ in ${\rm {Der}} (   A)$.
   
   Any   derivation  $d:{ A}\to { A}$ carries $A'=[A,A]$ into itself as  
$$d(xy-yx )=  d(x) y - yd(x)  +   x d(y) -d(y)x     $$
for  $x,y\in A$. 
Therefore~$d$  induces a linear endomorphism of $ \check A = A/ A'$.
A linear   endomorphism of $ \check A $ is   a \emph{weak  derivation} if it is induced by a derivation  ${ A}\to { A}$.

  By \cite[Lemma  4.4]{Cb}, for any derivation $d:A\to A$ and any integer $n\geq 1$, there is a unique  derivation $\widetilde d:A_n\to A_n$ such that $\widetilde d (a_{ij})= (d(a))_{ij}$ for all $a\in A, i,j\in \{1,..., n\}$. Indeed, this formula defines~$\widetilde d$ on the generators of the algebra~$A_n$;   the compatibility with the defining  relations    is straightforward.  Clearly,   $\widetilde  d ({\rm {tr}} (a)) ={\rm {tr}} (d(a))$ for all $a\in   A$.  Therefore $\widetilde d(A^t_n) \subset A^t_n$ and the restriction of $\widetilde d$   to $A^t_n$ is a derivation of the algebra $A^t_n$
   
 \subsection{Remark}    If $C^\infty(M)$ is the algebra of smooth $\RR$-valued functions on a smooth manifold~$M$, then each smooth vector field~$v$ on~$M$ induces a derivation $d_v$  of $C^\infty(M)$ carrying a   function $f\in  C^\infty(M)$ to the function $df(v):M\to \RR$. The map $v \mapsto d_v$ defines  a Lie algebra isomorphism from the Lie algebra of smooth vector field on~$M$ (with the Jacobi-Lie bracket) onto  ${\rm {Der}} (C^\infty(M))$. Given an algebra~$A$ and an integer $n\geq 1$, these results suggest to view    the derivations of the trace  algebra $A^t_n$ as vector fields on (the smooth part) of the affine quotient scheme  ${\rm Rep}_n (A)//G$.  More generally, tensor fields on this affine scheme may be defined as maps $\{(A^t_n)^m\to A^t_n\}_{m\geq 1}$ which are derivations in all~$m$ variables.  Here   for a set~$E$ and an integer $m\geq 1$, we let  $E^m$ be  the  direct product of~$m$ copies of~$E$.

    \section{Braces and brackets}

We define and study braces. 

\subsection{Braces}\label{Bracket modules}
  For an integer $m\geq 1$, an    \emph{$m$-brace}   in an algebra~$A$ is  a mapping  $\mu: (\check A)^{m} \to \check A$   which is a weak  derivation  in all~$m$ variables:     for any $1\leq j\leq  m$ and   $x_1,..., x_{j-1}, x_{j+1},...,x_{m} \in  \check A$, the     map $$\check A \to \check A, \, \, x\mapsto \mu (x_1, ... , x_{j-1} , x , x_{j+1} , ... ,x_m)$$ is a weak derivation.  In particular,   $\mu$ has to be linear in all variablres.   For $m=1$, an $m$-brace in~$A$ is just a weak derivation   $\check A \to \check A$. 
  
If~$A$ is a  commutative algebra, then $A'=0$,   $\check A=A$, and  an  $m$-brace in~$A$  is  a mapping  $\mu: A^{m} \to   A$   which is a    derivation  in all  variables:   for any $1\leq j\leq  m$ and any $x_1,..., x_{j-1}, x_{j+1},...,x_{m} \in  A$, the   map $$A \to A,  \, \, x\mapsto \mu (x_1,  ... , x_{j-1} , x  ,x_{j+1} , ... , x_mt)$$ is   a derivation. 

 The following lemma - inspired by W. Crawley-Boevey \cite{Cb} - is our main tool producing braces in the trace algebras.
  
\begin{lemma}\label{Cbtlemmmmheor}  For any integers $m, n\geq 1$ and any   $m$-brace $\mu$ in an algebra~$A$, there is a unique   $m$-brace $\mu_n$ in the algebra  $ A_n^t$ such that   the trace ${\rm {tr}}: \check A\to A^t_n$ carries~$\mu$ to $\mu_n$ that is for all $x_1,..., x_m\in \check A$, we have  \begin{equation}\label{ccc} \mu_n( {\rm {tr}} (x_1), ... , {\rm {tr}} (x_m))= {\rm {tr}} ( \mu(x_1 , ..., x_m) ). \end{equation}
\end{lemma}

  \begin{proof} The uniqueness of $\mu_n$ is clear  as ${\rm {tr}} (A)$ generates the algebra $A^t_n$.    We first prove the existence of $\mu_n$  for $m=1$.   We need to show  that given a weak derivation $\mu:\check A\to \check A$, there is a   derivation $\mu_n :A^t_n \to A^t_n$ such that   $\mu_n( {\rm {tr}} (x))= {\rm {tr}} (\mu(x))$ for all $x\in \check A$.   Pick a  derivation $d:A\to A$ inducing $\mu$.  By Section~\ref{coordrerdd2w},  the induced   derivation $\widetilde d:A_n\to A_n$ restricts to a derivation $\mu_n: A^t_n\to A^t_n$  of the algebra $A^t_n$.  The map $\mu_n$ satisfies the  conditions of the lemma.

Suppose now that $m\geq 2$. Since the  algebra $A_n^t$ is generated by the set ${\rm {tr}} (\check A) \subset A^t_n$, every  $y\in A^t_n$ has a (non-unique)    finite expansion  
\begin{equation}\label{lblb} y=\sum c_{x_1^y,..., x_r^y} {\rm {tr}} (x_1^y) \cdots {\rm {tr}} (x_r^y) \end{equation}
  where the sum is over some  finite sequences   $x_1^y,..., x_r^y\in \check A$ and the coefficients 
  $c_{x_1^y,..., x_r^y}$ are in~$R$. 
  Pick   any $y_1,..., y_m \in  A^t_n$ and for   $j=1,..., m$ pick an expansion $E_j$ of $y_j$ as in \eqref{lblb}.   If  an $m$-brace $\mu_n: (A_n^t)^m\to A_n^t$  satisfies the conditions of the lemma, then  using  $E_1,..., E_m$,  the Leibnitz rule and \eqref{ccc}, we obtain  that  $\mu_n(y_1,  ... ,  y_m)= F(E_1,..., E_m) $ where $ F(E_1,..., E_m) $   is     a sum  of 
   products  determined by   the  summands on the right-hand sides of the expansions $E_1,..., E_m$.  Each product   involves    factors of 3 types: 
 
 (I) the coefficients $c  \in R$  appearing in the summands   in question;
 
 (II) the   traces  $ {\rm {tr}} (x_i^{y_j})\in A_n^t$ where $j=1,..., m$ and  $i$ runs over   the indices $1,..., r=r_j$  determined by the summand of $E_j$ except  one of these  indices, say, $i_j$;
 
(III)  the   factor $ {\rm {tr}} ( \mu(x_{i_1}^{y_1} , ... , x_{i_m}^{y_m}) ) \in A_n^t$.

   We claim that the element  $F=F(E_1,..., E_m) $ of $A_n^t  $    does not depend on the choice of  the expansion $E_m$ of  $y_m$. It is easy to reduce this claim to its  special case where $y_j={\rm {tr}} (x_j)$ for $j=1,..., m-1$ and $x_j\in \check A$. (It is understood that we keep $E_m$ and  use the formula  $y_j={\rm {tr}} (x_j)$  as  the expansion $E_j$ for $j\leq m-1$). Consider the projection 
  $p:A\to \check A=A/A'$.    By the assumptions of the lemma,   there is    a derivation $d:A\to A$ (possibly,  depending  on $x_1,..., x_{m-1}$)  such that     $$\mu (x_1,  ... ,  x_{m-1} ,p(a) ) =
 p(d(a)) \in \check A\quad {\rm {for\,\,  all}} \,\, a\in A. $$  
For any $x\in \check A$ and $a\in p^{-1} (x)\subset A$, we have the following equalities in $A_n^t$:  $${\rm {tr}} ( \mu(x_{1} , ... ,  x_{m-1} , x) )= {\rm {tr}} ( \mu(x_{1},  ... , x_{m-1} ,  p(a) )) $$ $$  = {\rm {tr}} (p(d(a)) )
 = {\rm {tr}} (d(a) ) =\sum_{i=1}^n (d(a))_{ii} =\sum_{i=1}^n \widetilde d( a_{ii})= \widetilde d({\rm {tr}}(a)) = \widetilde d({\rm {tr}}(x)). $$
 Using this formula to compute all  factors of type (III) above,  we easily deduce that     $F=  \widetilde d(y)$. Since $ \widetilde d(y)$ does not depend on the choice of $E_m$, neither does~$F$. 
 
 %This proves our claim above.
 
 Coming back to  arbitrary $y_1,..., y_m \in A^t_n$, we similarly prove that the element $F=F(E_1,..., E_m)$ of $ A_n^t$     does not depend on the choice of  the expansion   $E_i$ for all  $i=1,..., m$. In other words, $F$ depends only on  $y_1,..., y_m$. We take~$F$ as $\mu_n(y_1,  ... ,  y_m)$. 
The resulting map $\mu_n: (A_n^t)^m\to A_n^t$  is easily seen to be an $m$-brace in  $ A_n^t$  and to satisfy \eqref{ccc}.
   \end{proof}

  \subsection{Brackets}\label{Bracketsssy modules}    Given an integer $m\geq 1$, an   \emph{$m$-bracket} in a module~$M$ is   a   map  $ \mu:{M}^{m}\to {M}$ which is linear in every variable.  
For    $\varepsilon\in R$,    we say that~$\mu$     is \emph{$\varepsilon$-symmetric} if  for all  $x_1,..., x_m\in {M}$, 
$$\mu(x_1 , \ldots ,  x_{m-1},x_m)= \varepsilon \,  \mu(x_m , x_1, \ldots , x_{m-1}  ).$$
%A similar  definition applies to algebra brackets. 
If $\varepsilon=+1$, then  $\varepsilon$-symmetric brackets are said to be \emph{cyclically symmetric}.
If $\varepsilon=-1$, then   $\varepsilon$-symmetric brackets are said to be \emph{skew-symmetric}.

%  \subsection{Symmetries}\label{Terprpsuuusrpnology} 

%    Clearly,   if an   $ \varepsilon$-symmetric  $n$-ary bracket in an algebra  is a   derivation  in one variable  %then it  is   a   derivation in all~$n$ variables. 

\begin{lemma}\label{Cbeeeetlemmmmheor}  Given an algebra~$A$,  integers  $m,n \geq 1, \varepsilon \in R$,   and an  $\varepsilon$-symmetric $m$-bracket $\mu$ in~$\check A$ which  is   a  weak   derivation in the $m$-th variable, there is a unique   $m$-brace $\mu_n$ in the algebra  $ A_n^t$ such that    the trace ${\rm {tr}}: \check A\to A^t_n$ carries~$\mu$ to $\mu_n$.  The brace $\mu_n$ is  $\varepsilon$-symmetric. 
\end{lemma}

  \begin{proof}  Since~$\mu$ is     $ \varepsilon$-symmetric  and    is a  weak derivation  in one  variable, it is a  weak derivation  in  all variables. Thus,~$\mu$ is a brace. By Lemma~\ref{Cbtlemmmmheor}, there is a unique  $m$-brace $\mu_n$ in   $ A_n^t$ such that   the trace   carries~$\mu$ to $\mu_n$. The $ \varepsilon$-symmetry of~$\mu$ implies 
  that $\mu_n$ is  $\varepsilon$-symmetric. 
  \end{proof}

\section{Braces in group algebras}\label{Derivative brackets in group algebras}

In  this section,   $A=R[\pi]$ is the group algebra of a group~$\pi$. We    construct   braces in~$A$ starting from Fox derivatives in~$A$. 

 \subsection{Computation of    $\check A$}\label{zzz}  By definition, the module $A' \subset A$ is generated by  the set $\{ uv-vu\, \vert \, u,v \in A\}$. Since $\pi\subset A$ generates~$A$,  the module $A' $ is generated by the set $\{ uv-vu\, \vert \, u,v \in \pi\}$.  Since   $uv=u(vu)u^{-1}$ for   $u,v\in \pi$, the module $A' $ is generated by  the set $\{uw u^{-1}-w \, \vert \, u,w \in \pi\}$.  Thus,   $\check A =A/A'=R\check \pi$ is the free module   whose basis~$\check \pi$  is the set of conjugacy classes of elements of~$\pi$.  
  
  \subsection{Fox derivatives}\label{Constructions of brackets}   A (left)  \emph{Fox derivative} in~$A$ is a linear map $\partial:A\to A$ such that $\partial (xy)=\partial (x) +x \partial (y)$ for all $x,y \in \pi \subset A$.
For any $x,y \in A$, we have then $\partial (xy)=\partial (x) \, {\rm {aug}} (y)+x \partial (y)$ where 
  ${\rm {aug}}:A\to R$ is the linear map carrying all elements of~$\pi$ to $1\in R$.
 For   $x\in \pi$,  we can uniquely expand $\partial (x)= \sum_{a\in \pi}    (x/a)_\partial \, a$ where $(x/a)_\partial\in R$ is non-zero  for a finite set of~$a$. Consider the map
   $$ \pi\to A, \,\, x\mapsto   \sum_{a\in \pi}   (x/a)_\partial \,  a^{-1} x a   $$
 and denote its linear extension $A\to A$ by  $ \Delta_{\partial}$.  
  
  \begin{lemma}\label{sDDDtrbucte}    
   $\Delta_\partial (A')=0$.
   \end{lemma}
  
\begin{proof}   
  It suffices to prove that $\Delta_\partial (xy  -yx)=0$ for any $x,y\in \pi$.
   We have  $$\partial (xy)=\partial (x)  +x \partial (y)= \sum_{a\in \pi}  \big (  (x/a)_\partial \, a+   (y/a)_\partial \, xa\big ).$$
   Therefore, by the definition of $\Delta_\partial$,  
  $$\Delta_\partial (xy  )= 
  \sum_{a\in \pi} \big ( (x/a)_\partial \,  a^{-1} xy a +    (y/a)_\partial \,  (xa)^{-1} xy (xa) \big )$$
  $$ =\sum_{a\in \pi} \big ( (x/a)_\partial \,  a^{-1} xy a +    (y/a)_\partial \,  a^{-1} y xa \big ).$$
  The  latter  expression is invariant under the permutation $x \leftrightarrow y$. So,  $\Delta_\partial (xy  )=\Delta_\partial (yx  )$ and   $\Delta_\partial (xy  -yx)=0$.
  \end{proof}

  The linear map $\check A=A/A'\to A$ induced by $\Delta_\partial:A\to A$ is denoted by $\check \Delta_\partial$.

 \begin{theor}\label{sDDjkgfjfgjDtrbucte} Let $p:A\to \check A$ be the projection.    For any 
   $m\geq 1$  and any Fox derivatives $\partial_1,  ..., \partial_m:A \to A$,  the  map  $\mu^m:\check A^{  m} \to \check A$ defined by  
\begin{equation}\label{cccdmm} \mu^m(x_1 ,     \ldots ,  x_m)=p \big  ( \check \Delta_{\partial_1} (x_1)    \cdots
  \check \Delta_{\partial_m} (x_m) \big ) \end{equation}
  for $x_1,   ..., x_m  \in  \check A$  is  an $m$-brace  in~$A$.
   \end{theor}

\begin{proof}   % Since the right-hand side of  \eqref{cccdmm} is linear in all~$m$ variables, it  does define   a linear map  $\mu^m:\check A^{ m} \to \check A$.  
 We need to prove that  $\mu^m$  is a weak derivation in all variables, i.e.,  for any  $i=1,...,m$  and $x_1,..., x_{i-1}, x_{i+1},..., x_m\in \check A$, the  map
\begin{equation}\label{678} \check A \to \check A, \,\, x \mapsto \mu^m(x_1,  \ldots   , x_{i-1} , x   ,  x_{i+1}, \ldots   , x_{m})\end{equation} is induced by a derivation in~$  A$. 
   Set $$G= \check \Delta_{\partial_1} (x_1)   \cdots
  \check \Delta_{\partial_{i-1}} (x_{i-1}) \in A \quad {\rm {and}} \quad H= \check \Delta_{\partial_{i+1}} (x_{i+1})   \cdots
  \check \Delta_{\partial_{m}} (x_{im}) \in A.$$
   For $x  \in \pi$, we expand $\partial_i (x)= \sum_{a\in \pi}  (x/a) a$ with $(x/a) =(x/a) _{\partial_i}$. Then 
   %$$\Delta_{\partial_m} (x)=\sum_{a\in \pi}  (x/a) a^{-1} x a $$ and 
   $$\mu^m(x_1,  \ldots , x_{i-1} , x   ,  x_{i+1}, \ldots  ,  x_{m}) =p(  G   \check \Delta_{\partial_{i}} (x) H) $$
$$=p(  G \big ( \sum_{a\in \pi}  (x/a)     a^{-1} x a\big ) H) =p(    \sum_{a\in \pi}  (x/a)   G  a^{-1} x a H) =p( \sum_{a\in \pi} (x/a) a HG  a^{-1} x )$$ where we use that $p(   G  a^{-1} x aH) =p(  a  HG  a^{-1} x )$.
   Thus,  the  map \eqref{678}  is induced by the linear map $ A\to A$ carrying any $x\in \pi$ to $    \sum_{a\in \pi} (x/a) a GH  a^{-1} x.  $ It remains to prove that  for any $F\in A$, the linear map $ d=d_F:A\to A$ carrying any $x\in \pi$ to $    \sum_{a\in \pi} (x/a) a F  a^{-1} x  $ is a derivation.
   Indeed, for
$x, y \in \pi$, we have 
$$d (x)=  \sum_{a\in \pi} (x/a) a F  a^{-1} x\,\, \,\, \,  {\rm {and}} \,\,\,    \quad d(y)=\sum_{a\in \pi} (y/a)  a F a^{-1} y.$$ Also, 
$$\partial_i(xy) =\partial_i(x) +x \partial_i (y)= \sum_{a\in \pi}  \big ( (x/a) a +   (y/a) x a \big )$$  
    and so 
   $$ d (xy)=\sum_{a\in \pi}  \big ( (x/a) a F a^{-1} xy  +   (y/a)  xa F (x a)^{-1} xy \big )   =d(x) y +x d(y).$$
     Thus, $d$ is a derivation in~$A$. This completes the proof of the theorem. \end{proof}  
     
For  $m=1$, Theorem~\ref{sDDjkgfjfgjDtrbucte} may be rephrased by saying that for any Fox  derivative~$\partial $ in~$A$, the linear map $\check A\to \check A$ induced by  $\Delta_\partial:A\to A$ is  also induced by a derivation  $d=d_\partial:A\to A $. This derivation   carries  any $x\in \pi$ to $\sum_{a\in \pi} (x/a)_\partial \, x ={\rm {aug}} (\partial (x)) x $. In contrast to  $\Delta_\partial$,  the derivation $d_\partial$ may not annighilate $A'$. % and  is insufficient  to treat the case $m\geq 2$.
   
   Combining Theorem~\ref{sDDjkgfjfgjDtrbucte} with  Lemma~\ref{Cbtlemmmmheor}  we obtain the following.

\begin{corol}\label{sDDjmmmkgfjfgjDtrbucte}    For any 
  integers  $m,n\geq 1$  and  Fox derivatives $\partial_1, ..., \partial_m$ in~$ A$,  there is a unique 
$m$-brace $\mu^m_n$ in $ A_n^t$ such that  for all  $x_1,  ..., x_m \in A$,  we have
 $$\mu^m_n({\rm tr} (x_1) ,    \ldots  , {\rm tr} ( x_m))={\rm tr} (    \Delta_{\partial_1} (x_1)      \cdots
   \Delta_{\partial_m} (x_m)   ) . $$
 \end{corol}
   
If  $\partial_1=  \cdots = \partial_m$, then the $m$-braces $\mu^m$ and $\mu^m_n$ are cyclically symmetric. This follows from the identities $p(xy)=p(yx)$ and 
   ${\rm tr} ( xy)={\rm tr} ( yx)$ for all $x,y\in A$.

\subsection{Equivalence of Fox derivatives}\label{topdequivolsett}  Given  a Fox derivative $\partial$ in~$ A$ and any $g\in \pi$, the linear map $A\to A, x \mapsto \partial (x)g$ is also a Fox derivative denoted $\partial \cdot g$. We say that  two  Fox derivatives  $\partial, \partial' $ in~$A$ are equivalent if there is $g\in \pi$ such that $\partial'=\partial \cdot g$. This is indeed an  equivalence relation.
% in the set of Fox derivatives in~$A$. 
Moreover,  equivalent Fox derivatives   induce the same  braces in $\check A$ and $A^t$. 
This follows from the identities $$\Delta_{\partial \cdot g} (x)=  g^{-1} \Delta_{\partial } (x)  g, \quad
 p(g^{-1}x g)=p(x), \quad {\rm tr}(g^{-1}x g)  ={\rm tr} (x)$$ for all $x\in A$ and $g\in \pi$.

\section{Quasi-Lie brackets and brace algebras}\label{Quasi-Lie brackets and brace algebras}

We define  quasi-Lie  brackets and brace algebras. 

\subsection{Quasi-Lie  brackets}
A \emph{quasi-Lie pair of brackets}  in  a module~$M$ is a pair formed by a  skew-symmetric 2-bracket $[-,-]$ in~$ M$ and a cyclically symmetric 3-bracket
$[-,-,-]$ in~$M$ such that
  for any $x,y,z \in M$, we have 
\begin{equation}\label{Jaco} [[x,y],z] + [[y,z], x]+ [[z,x], y]= [x,y,z]- [y,x,z] . \end{equation} 
%Note for the record that  $[x,y]=- [y,x]$ and $[x,y,z]= [y,z,x]$  for all  $x,y,z \in M$. 
Here the left-hand side  is the usual  Jacobiator of the 2-bracket. Both sides of \eqref{Jaco}  are   cyclically symmetric.
 We call   \eqref{Jaco} the \emph{quasi-Jacobi identity}. 
For the zero  3-bracket,  we recover the  standard Jacobi identity.

%A \emph{bracket homomorphism} from a bracket module~$M$ to a bracket module~$N$ is a  linear map  $f:M  \to N$ such that
%$[f(x), f(y)] =f([x,y]) $ for all $x,y \in M$  {and}  $[f(x), f(y), f(z)]  =f([x,y,z]) $
%for all $x,y,z\in M$.

  \subsection{Examples} 1.  Any bilinear pairing $M^{ 2} \to M, (x,y)\mapsto xy$ induces a  quasi-Lie pair of brackets in~$M$ with the 2-bracket  $[x,y]=xy-yx$ and the 3-bracket 
$$\label{333} [x,y,z]=(xy)z +(yz)x+ (zx)y -  x(yz) - y(zx) - z(xy) $$
  for  $x,y,z\in M$. 
  
  2. For a quasi-Lie pair of brackets $[-,-], [-,-,-]$  in a module~$M $  and a   3-bracket $b$ in~$M$  invariant under all permutations of the variables,  the pair   $[-,-], [-,-,-]+b$ is also a  quasi-Lie pair.

  \subsection{Brace algebras}\label{Teadsrmy}  
A   \emph{brace algebra} is an algebra~$A$  endowed with a quasi-Lie pair of brackets in  the module  $\check A$ such that both these brackets are braces  in~$A$ in the sense of Section~\ref{Bracket modules}. 
%Brace algebras with zero 3-bracket   were introduced   by  Crawley-Boevey \cite{Cb} who calls them   $H_0$-Poisson algebras.
For example, a commutative brace   algebra  with zero 3-bracket is a  Poisson algebra in the usual sense.

A   \emph{brace homomorphism} from a  brace algebra~$A$ to a  brace algebra~$B$ is a bracket-preserving linear map $f:\check A \to \check B$. Thus, $f$ should satisfy  $[f(x), f(y)]=f([x,y])$ and $[f(x), f(y), f(z)]=f([x,y,z])$ for all $x,y,z\in \check A$.

\begin{lemma}\label{Cbtlemyahhhmmmheor} Let $A$ be a commutative algebra  carrying a  skew-symmetric   2-brace
$[-,-]$  and a cyclically symmetric   3-brace  $[-, -,-]$.   If    \eqref{Jaco} holds for all elements of a generating set of~$A$, then it holds for all elements of~$A$.    
\end{lemma}

  \begin{proof}  Let $L(x,y,z)$ and  $R(x,y,z)$ be respectively the left  and the right  hand-sides of \eqref{Jaco}.   Since    $L(x,y,z)$ and $R(x,y,z)$ are linear in $x,y,z$ and cyclically symmetric, it suffices to verify the following: if  \eqref{Jaco} holds for the triples $x,y,z\in A$ and $x,y,t\in A$, then it holds for the triple $x, y, zt$.
  Since $[-,-]$ is a  brace, 
  $$ [[x,y],zt]  =  z  [[x,y],t] +  [[x,y],z] t, $$
   $$[[y,zt], x] =    [z [y,t],x] +[[y,z]t, x] $$ $$= z[  [y,t],x] +[z ,x][y,t]+ [y,z]   [t, x]  +[[y,z],x] t. $$
Similarly, 
 $$ [[zt,x], y] =   z[[t,x],y]+ [z,y][t,x] +[z,x][t,y]+ [ [z,x], y]t.$$
  Adding these three  expansions    and using the skew-symmetry of $[-,-]$, we get
    $$L(x,y,zt)= z L(x,y,t)+L(x,y,z)t. $$
Thus,~$L$  satisfies the Leibnitz rule  in the last variable.
 Since the   bracket $[-,-,-]$ also satisfies this rule, so does $R(x,y,z)=[x,y,z]-[y,x,z]$.  Consequently,   if  \eqref{Jaco} holds for the triples $x,y,z $ and $x,y,t $, then it holds for the triple   $x, y, zt$.
   \end{proof}

%In the context of trace algebras, our interest in brace algebras is due to the following theorem. 

 Recall the  trace algebras $\{A_n^t\}_{n\geq 1}$ associated with any algebra~$A$.
 
\begin{theor}\label{Cbtheor}  For any brace algebra~$A$ and    integer  $n\geq 1$, there is a unique  brace algebra   structure on  $A_n^t$ such that    ${\rm {tr}}:  \check A  \to A^t_n$  is a brace homomorphism. 
\end{theor}

\begin{proof} 
     Let $[-,-]$ and $ [-,-,-]$ be the    brackets  in~$\check A$ forming a quasi-Lie pair.  By Lemma~\ref{Cbeeeetlemmmmheor}, there are unique       braces  $[-,-]^t$ and  $[-,-,-]^t$   in   $ A_n^t$  such that  $$[{\rm {tr}} (x), {\rm {tr}}(y)]^t={\rm {tr}} ( [x,y])\quad  {\rm {and}} \quad [{\rm {tr}} (x), {\rm {tr}}(y), {\rm {tr}} (z)]^t={\rm {tr}} ([x,y, z])$$ for all $x,y, z\in \check A$.   
Then    \eqref{Jaco} holds  for all  elements of the  set $ {\rm {tr}} (\check A) \subset  A^t_n$. Since this set generates $A^t_n$,  Lemma~\ref{Cbtlemyahhhmmmheor} implies that \eqref{Jaco} holds for all elements of $  A^t_n$. Also, since  $[-,-]$ is $(-1)$-symmetric and $ [-,-,-]$   is   $1$-symmetric, so are the braces   $[-,-]^t$ and   $[-,-,-]^t$. Thus, these braces form a quasi-Lie pair.           This turns  $A^t_n$  into a brace algebra satisfying the conditions  of the theorem. \end{proof}

 \subsection{Remark} A  bracket  in a module   is     fully symmetric  if it is invariant under all permutations of the variables. 
A  quasi-Lie pair of brackets    in a module~$M $  gives rise to  a fully symmetric 3-bracket $s:M^{  3}\to M$    by
\begin{equation}\label{Jacocbcb}s(x,y,z)= 2 [x,y,z] - [[x,y],z] - [[y,z], x]- [[z,x], y] \end{equation} 
for any $x,y,z\in M$. The cyclic symmetry of~$s$ is obvious  and the invariance of $s (x,y,z)$ under the permutation $x\leftrightarrow y$   follows from \eqref{Jaco}. 
Conversely, if 2 is invertible in~$R$, then   we can recover the 3-bracket from $[-,-]$  and~$s$ via \eqref{Jacocbcb}. Formula~\eqref{Jaco} follows then from the identity $s (x,y,z) =s(y,x,z)$.     This establishes a bijective correspondence between  quasi-Lie pairs of brackets in~$M$  and  pairs
(a skew-symmetric 2-bracket in~$M$, a fully symmetric 3-bracket in~$M$). 

%\vspace{1mm}
%\part*{\quad \quad \quad \quad \quad \quad \quad \quad \quad    Part II. Gates and quasi-surfaces}  
%\vspace{2mm}
  
\section{Topological gates}\label{A topological example}

We  define gates in    topological spaces and show how they give rise to braces.

\subsection{Gates}\label{topolsett} 
  A \emph{cylinder neighborhood} of a subset $C$ of a topological space~$X$  is   a  pair consisting of a closed  set $U\subset X$ with $C \subset U$ and  a homeomorphism  $ U\approx C\times [-1,1]$     carrying~$C $ onto  $C\times \{0\}$ and carrying ${\rm {Int}} (U)$ onto  $C\times (-1, 1)$. Note that then $C \subset  {\rm {Int}} (U)$ and~$C$ is closed in~$X$. A \emph{gate} in~$X$ is a path-connected  subspace $C \subset X$ endowed with a cylinder neighborhood in~$X$ and such that all loops in~$C$ are contractible in~$X$. An example of a gate is provided by a   simply connected  codimension 1   proper submanifold~$C$ of a manifold together with a suitable homeomorphism of a closed neighborhood of~$C$   onto $C\times [-1,1]$. 

For the rest of this section, we fix   a   path-connected topological space~$X$, a gate $  C\subset X$,  and its cylinder neighborhood $U\subset X$ which we identify with $ C\times [-1,1]$ so that $C =C\times \{0\}$. Pick a point $\ast  \in X \setminus U$ and  set
 $\pi=\pi_1(X, \ast)$.

\subsection{Gate derivatives} Here we       associate with the gate~$C$  an equivalence class of Fox derivatives   in the algebra $A =R[\pi]$. We start with preliminaries on (continuous) paths.  Let $q:X \to S^1=\{z\in \CC\, \vert \, \vert z\vert=1\}$ be the  map which   carries  $C\times \{t\} \subset U $ to ${\rm{exp} }(\pi i t)\in S^1$ for all $t\in [-1,1]$  and carries  $X \setminus U $ to $-1\in S^1$.
  We say that a path $a:[0,1]\to X$ is \emph{transversal to~$C$} if $a(0), a(1) \in X \setminus C$ and  the map $qa:[0,1] \to S^1$ restricted to $(0,1)$  is transversal to $1  \in S^1$. Then  $a^{-1}(C) =(qa)^{-1}(1)$ is a finite subset of $(0,1)$.  For a path~$a$, we denote the inverse path by~$\overline a$.  A path $a:[0,1]\to X$ is a \emph{loop based in} $\ast$ if $a(0)=a(1)=\ast$. Such a loop~$a$ represents an element of~$\pi$ denoted $[a]$. 
  
 Pick   a path $\gamma:[0,1]\to X$ such that $\gamma (0)=\ast$   and $\gamma ( 1) \in C$.
Consider a   loop $a:[0,1]\to X$ based in $\ast$ and  transversal to~$C$.  For   $t\in a^{-1}(C) \subset (0,1) $,  we  let $a_t^\gamma$ be the path  in~$X$
 obtained as the product of the path $a\vert_{[0,t]}$  with any path~$\beta$  in~$C$ from $a(t) $  to  $\gamma ( 1) $, and finally with  $ \overline \gamma$. Then $a_t^\gamma$ a loop    based in~$\ast$. Since all loops in~$C$ are contractible in~$X$, the homotopy class $[a_t^\gamma] \in \pi$ does not depend on the choice of~$\beta$. Set $\varepsilon_t(a)=1$ if at $a(t)\in C$ the loop~$a$ crosses~$C$ upwards (i.e., from $C\times [-1,0) $ to $C \times (0,1]$), and   $\varepsilon_t(a)=-1$ otherwise.
  Set 
  \begin{equation}\label{tweeotwo3} \partial_C^\gamma(a)= \sum_{t \in  a^{-1}(C)} \varepsilon_{t}(a) \, [a_t^\gamma] \in  R [\pi] =A.   \end{equation}

  \begin{lemma}\label{svvvvtrbucte}  Formula \eqref{tweeotwo3} defines a  map $\pi \to A$ whose linear extension $ A\to A$,  denoted $\partial_C^\gamma$,  is a Fox derivative.  If $\gamma':[0,1]\to X$ is another path from $ \ast$   to~$  C$, then the Fox derivatives  $\partial_C^\gamma$ and  $\partial_C^{\gamma'}$ are equivalent in the sense of Section~\ref{topdequivolsett}. 
\end{lemma}
  
\begin{proof} It is clear that   all elements of~$\pi$  can be represented by loops based in $\ast$ and transversal to~$C$.  We claim that if two such  loops $a,a'$  are homotopic, then 
 $\partial_C^\gamma(a)=\partial_C^\gamma(a')$. There is a homotopy $(a_u)_{u\in [0,1]}$ from $a=a_0$ to $a'=a_1  $ such that the loop $a_u$ is based in~$\ast$ and transversal to~$C$ except for a finite set of $u\in (0,1)$ near which  the homotopy pushes  a branch of $a_u $  across~$C$ creating or destroying a pair of transversal  crossings   with~$C$.  It is easy to see that the contributions of these two crossings to 
$\partial_C^\gamma(a_u)$  cancel each other. Therefore,  Formula \eqref{tweeotwo3} yields a well-defined map $ \pi \to A$ which extends by linearity to a map $\partial_C^\gamma:A \to A$. 

If   $a,b$ are  loops in~$X$ based in~$\ast$ and transversal to~$C$, then so is their product,  and it follows directly  from the definitions that 
$\partial_C^\gamma(ab)= \partial_C^\gamma(a) + a \partial_C^\gamma(b)$. Consequently,  $\partial_C^\gamma$    is a Fox derivative in~$A$. 

Given two   paths $\gamma, \gamma'$ from~$\ast$ to~$C$, we let  $g \in \pi$ be the homotopy class of the loop  obtained as the product of  $\gamma$ with a path  in~$C$   from $\gamma (1) \in C $  to  $\gamma' (1) \in C $, and with  $\overline {\gamma'}$. It is easy to see  that  $\partial_C^{\gamma'}  = \partial_C^\gamma \cdot  g$. Thus,  the Fox derivatives  $\partial_C^\gamma$ and  $\partial_C^{\gamma'}$ are equivalent. \end{proof} 

%We denote by $\partial_C$ the equivalence class of Fox derivatives produced by Lemma~\ref{svvvvtrbucte}.

 \subsection{Gate braces} \label{Gate brackets}   Let 
  ${\mathcal L} = {\mathcal L} (X)$  be  the set of free   homotopy classes of loops in~$X$  
 and let $ R{\mathcal L}$  be  the  free module with basis~${\mathcal L}$.  The   map $  \pi \to \mathcal L$ carrying  the homotopy classes of  loops to their free homotopy classes    induces  a bijection $\check \pi \approx  \mathcal L$ where $\check \pi$   is the set of conjugacy classes of elements of~$\pi$. 
  By Section~\ref{zzz},  
 $\check A=A/A'  = R\check \pi $ so that we can identify  
$\check A $ with   $R\mathcal L$.  

 By  Lemma~\ref{svvvvtrbucte}, a    sequence of $m\geq 1$ gates in $  X\setminus \{\ast\} $ (not necessarily disjoint or   distinct) determines a sequence  of $m\geq 1$ equivalence classes of Fox derivatives in~$A$. By   Section~\ref{Derivative brackets in group algebras}, the latter  induces     $m$-braces in the algebras~$A$ and   $\{A_n^t\}_{n\geq 1}$. 
 In particular,  the sequence of $m\geq 1$ copies of a gate $ C \subset X$ determines a cyclically symmetric  $m$-brace $\mu^m_C:{\check A}^{ m} \to \check A$ in~$A$. We   compute  $\mu^m_C$ in geometric terms as follows.
 Consider  $m$   loops $a_1,..., a_m: [0,1]\to X$ based in~$\ast$ and  transversal to~$C$. Pick a  point $\star \in C$.  For   $i=1,...,m$ and $t\in a_i^{-1}(C) \subset (0,1)$,   let $a_{i,t}$ be the loop based in $\star$ and obtained as the product of a  path $\beta_{i,t}$ in~$C$   from $\star$ to $a_i(t)\in C$, the loop  $a_i\vert_{[ t, 1]}\, a_i\vert_{[0,t]}$  based in $a_i(t)$, and the path $\overline {\beta_{i,t}}$.
   In the next  lemma, the free homotopy class of a loop~$b$ in~$X$ is denoted  $\langle b \rangle$.

\begin{lemma}\label{strleleucte}  Under the assumptions above, 
  \begin{equation}\label{twotwo3aa} \mu^m_C (\langle a_1\rangle , \ldots ,  \langle a_m\rangle) =
 \sum_{t_1\in a_1^{-1}(C), \ldots, t_m\in a_m^{-1}(C)} \prod_{i=1}^m \varepsilon_{t_i}(a_i) \langle \prod_{i=1}^m a_{i,t_i} \rangle \in R\mathcal L  .   \end{equation}
\end{lemma}

\begin{proof}   We claim that both sides of \eqref{twotwo3aa} are preserved  when  the loop  $a_1$ is replaced by the loop  $a'_1=ba_1 \overline {b}$  for a  path $b:[0,1]\to X$  transversal to~$C$ and such that  $b(1)=a_1(0)$. The  invariance of the left-hand side  is obvious since $\langle a'_1\rangle=\langle a_1\rangle$. 
We prove the  invariance of the right-hand side which we denote by $\sigma (a_1,..., a_m)$.
If the path~$b$ misses~$ C $, then the claim    is obvious because  the loops $a_1$ and $a'_1$ meet~$C$ in the same points which contribute the same  to $\sigma (a_1,..., a_m)$ and $\sigma (a'_1,   ..., a_m)$. Otherwise, the path~$b$   expands as a product of a finite number of paths each of which is transversal to~$C$ and intersects~$C$ in  one point. Thus, it suffices to prove our claim in the case where~$b$ intersects~$C$ in  one point, say, $c$. Then   $a'_1=ba_1 \overline {b}$ meets~$C$  at  the crossings of $a_1$ with~$C$ and  two additional crossings at the point~$c$ which is  traversed first by~$b$ and then  by $\overline b$. Let $0<u <w <1$ be the corresponding values of the parameter, so that   $a'_1(u)=c=a'_1(w)$ 
are respectively  the first  and  the last crossings  of $a'_1$ with~$C$. It is  easy to see  that $\varepsilon_{u}(a'_1)= - \varepsilon_{w}(a'_1)$ and  the corresponding loops $a'_{1, u}$ and $a'_{1, w}$ are   homotopic. Therefore the terms
of $\sigma (a'_1,  ..., a_m)$  associated with $t_1=u$ and $t_1=w$ cancel each other, while the   remaing terms  yield $\sigma (a_1,..., a_m)$. Thus, $\sigma (a_1,..., a_m) = \sigma (a'_1,   ..., a_m)$.

Replacing   $a_1$ by $ba_1\overline b$ for  a path~$b$     transversal to~$C$ and  running from~$\ast$ to $a_1(0)$, we obtain a loop  transversal to~$C$ and  based in~$\ast$. By the previous paragraph, it suffices to prove   \eqref{twotwo3aa}  with this new loop instead of $a_1$. By a similar argument,  it suffices to prove   \eqref{twotwo3aa}  in the case where all the loops $a_1,..., a_m$ are based in~$\ast$.

Now, pick     a path $\gamma:[0,1]\to X$ from  $\gamma (0)=\ast$   to  $\gamma ( 1) =\star \in C$.
By the definition of the Fox derivative $\partial= \partial^\gamma_C$,  for any  $i=1,..., m$,  
$$ \partial ([a_i])= \sum_{t \in  a_i^{-1}(C)} \varepsilon_{t}(a_i) \, [\eta_{i,t}] \in  R [\pi] =A$$
where     $\eta_{i,t}= a_i\vert_{[0,t]} \, \overline {\beta_{i,t}} \,  \overline \gamma$. 
Therefore 
$$   \Delta_\partial  ([a_i])= \sum_{t \in  a_i^{-1}(C)} \varepsilon_{t}(a_i) \,
[ \,   \overline {\eta_{i,t}} \, a_i \, \eta_{i,t} \,  ] =\sum_{t \in  a_i^{-1}(C)} \varepsilon_{t}(a_i) \,
[ \,   \gamma  {\beta_{i,t}} \overline {a_i\vert_{[0,t]}} \, a_i \,   a_i\vert_{[0,t]} \, \overline {\beta_{i,t}} \,  \overline \gamma \,  ] $$ $$=\sum_{t \in  a_i^{-1}(C)} \varepsilon_{t}(a_i) \,
[ \,   \gamma  {\beta_{i,t}}  a_i\vert_{[t,1]}    \,   a_i\vert_{[0,t]} \, \overline {\beta_{i,t}} \,  \overline \gamma \,  ] =\sum_{t \in  a_i^{-1}(C)} \varepsilon_{t}(a_i) \,
[ \,   \gamma  a_{i,t}   \overline \gamma \,  ] .$$
Setting $x_i=\langle a_i \rangle \in \check \pi$ for $i=1,..., m$ and substituting in Formula \eqref
{cccdmm} the above expression for $ \check  \Delta_\partial  (x_i)  = \Delta_\partial  ([a_i])$, we obtain a formula equivalent to \eqref{twotwo3aa}.
 \end{proof}

\subsection{The dual map}\label{The dual map} The  gate $C\subset X$ determines 
a linear map $v_C: H_1(X;R)\to R$ \lq\lq dual'' to~$C$. This map carries the homology class $[[a]]\in H_1(X; R)$ of a  loop $a:[0,1]\to X$ transversal to~$C$  to   $\sum_{t\in a^{-1}(C)} \varepsilon_t (a)$.  It is clear that the sum of the coefficients of  the expression \eqref{twotwo3aa} is equal to
$   \prod_{i=1}^m v_C ([[a_i]])$. 

%where $[[a_i]]\in H_1(X; R)$ is the homology class of the loop $a_i$. 

\section{Quasi-surfaces}\label{section2-}

\subsection{Generalities}\label{Terminology}\label{notat} By  a   \emph{surface} we mean  a  smooth 2-dimensional manifold  with   boundary.  %A \emph{(topological) circle} is   a  topological space homeomorphic to the unit circle $S^1\subset \CC$.     
% A \emph{(topological) segment} is   a  topological space homeomorphic to the unit segment $[0,1]\subset \RR$.     
A \emph{quasi-surface} is a  topological space~$X$ obtained by gluing a   surface~$\Sigma$ to a topological space~$Y$ along a continuous map   $ f: \alpha  \to Y$ where $ \alpha \subset \partial \Sigma $ is a union of a finite number of    disjoint  segments  in $\partial \Sigma$. 
%The structure of a quasi-surface in~$X$ includes  a  choice of $Y, \Sigma, \alpha, f$ as above.  %We say   that~$X$ %\emph{extends}~$
%\Sigma$. 
%We call~$Y$  the \emph{singular part} of~$X$ and call~$\Sigma$ the \emph{surface part} of~$X$. 
Note that   $Y\subset X$ and   $X\setminus Y =\Sigma\setminus \alpha$.  Here   we impose no conditions on~$Y$ and do not require~$\Sigma$ to be compact or connected or  maximal among   surfaces in~$X$.

  The quasi-surface~$X$ has path-connected components of 3 types: (i) components of~$\Sigma$ disjoint from~$\alpha$; (ii) path-connected  components of~$Y$ disjoint from  $f(\alpha)\subset Y$; (iii) path-connected  components of~$X$ meeting both~$\Sigma$ and $Y$.  For components of type (i) our results below are standard in the  topology of  surfaces. For components of type (ii),   all our operations are identically  zero. 
 The novelty  of this work concerns the components of type (iii).

\subsection{Examples}\label{Exampleskk} In the following   examples, $\Sigma$ is a surface. 

 1.  When    $\alpha$ is a family of $m\geq  1$ disjoint segments   in $\partial   \Sigma$,  the unique map from~$\alpha  $ to a 1-point space  determines   a   quasi-surface. For $m=1$, it    is a  copy of~$\Sigma$. As  a consequence, any surface with non-void boundary is a quasi-surface.

2. Given $m\geq 1$  disjoint finite subsets of $\partial \Sigma$, we  obtain  a quasi-surface by collapsing each of these subsets into a point.   Here $Y$ is an $m$-point set  with discrete topology and~$\alpha$ is a  small  closed neighborhood  in $\partial \Sigma$
  of the union of our  finite   sets.

3. Given  $m \geq 1$ disjoint   segments $\alpha_1,..., \alpha_m$  in  $  \partial \Sigma$ and~$m $ points $y_1,..., y_m$ in a topological space~$Y$,    we obtain a quasi-surface   by    gluing~$\Sigma$ to~$Y$ along the map carrying $\alpha_k$ to $y_k$ for  $k=1,...,m $.  

4.  Let $\Sigma_0$ be a surface  with boundary and let~$\alpha\subset \Sigma_0$ be  a  union of a finite number  of    disjoint  proper  embedded segments in~$\Sigma_0$. Suppose that~$\alpha$ splits~$\Sigma_0$ into   two  subsurfaces  (possibly disconnected)  $\Sigma\subset \Sigma_0 $ and $\Sigma'\subset \Sigma_0$  so that  $\alpha=\Sigma \cap \Sigma'=\partial \Sigma \cap \partial \Sigma'$. Then  $\Sigma_0$  is homeomorphic to the quasi-surface determined by the tuple~$(\Sigma, \alpha \subset \partial  \Sigma, Y=\Sigma',f)$   where   $f:\alpha \to \Sigma'  $ is    the inclusion.

 \subsection{Conventions}\label{Ternnnminology}
   Fix for the rest of the paper a tuple $X, Y, \Sigma, \alpha , f$   as in Section~\ref{Terminology}.  We    assume that~$X$ is path-connected, $\alpha \neq \emptyset$,  and~$\Sigma$ is oriented.
 We will  
  identify a closed    neighborhood  of~$\alpha$ in~$\Sigma$   with $\alpha \times [-2,1]$ so that 
  $\alpha=\alpha \times \{-2\} $ and  
  $$ \partial \Sigma  \cap (\alpha \times [-2,1])=(\alpha \times \{-2\} )  \cup (\partial \alpha \times [-2,1]).$$  We will often use   the surface
$$\Sigma'=\Sigma \setminus  (\alpha \times [-2,0)) \subset \Sigma \setminus \alpha \subset X   $$ which is  a  copy of $\Sigma$    embedded in~$X$. It is   called the \emph{surface part} of~$X$.   We   provide~$\Sigma'$  with the   orientation induced  from that of~$\Sigma$.

   For    $k\in  \pi_0(\alpha)$,      denote by    $ \alpha_k^\circ $ the corresponding   component of $\alpha\subset \partial \Sigma$. Set   $$\alpha_k=\alpha_k^\circ \times \{0\} \subset  \partial  \Sigma'  \subset X. $$  
Clearly, $\alpha_k$ is an  embedded segment in~$X$.    Endowing $\alpha_k$  with  the cylinder neighborhood $\alpha_k^\circ \times [-1,1] \subset     \Sigma \setminus \alpha \subset X $   we turn $\alpha_k$ into  a gate in~$X$ in the sense of Section~\ref{topolsett}. This is  the  \emph{$k$-th gate} of~$X$.
 The gates $\{\alpha_k\}_k$    split~$X$ into the  surface part  $\Sigma'$  and the \emph{singular part}
 which is   the mapping cylinder of the gluing map $f:\alpha\to Y$. All     paths from a point of $ \Sigma' $ to a point of $X \setminus \Sigma'$   have  to cross a  gate.

 \subsection{Gate orientations}\label{koko}\label{++kokoff}\label{Notation andww terminology}     A   \emph{gate orientation}  of~$X$     is   an orientation of all the gates $\{\alpha_k\}_k$ of~$X$. Gate orientations of~$X$  canonically correspond to orientations of the 1-manifold $\alpha \subset \partial \Sigma$.
%In the case of surfaces (i.e., when  $\alpha=Y=\emptyset$) an orientation of $X$ is just an orientation of~$\Sigma$.
Given a gate  orientation~$\omega$ of~$X$ and      points $p,q \in  \alpha_k$,  we  say that~$p$ \emph{lies on  the $\omega$-left of~$q$} and~$q$ \emph{lies on  the $\omega$-right of~$p$}  if $p\neq q$ and  the $\omega$-orientation  of $\alpha_k$   leads from~$p$ to~$q$. We write then  $p<_\omega q$ or $q>_\omega p$. 
We   set  $\varepsilon(\omega, k)=+1$ if the $\omega$-orientation of~$\alpha_k$ is compatible with the orientation of~$\Sigma$, i.e.,  if the pair (a $\omega$-positive tangent vector  of  $\alpha_k \subset \partial \Sigma'$, a   vector  directed inside~$\Sigma'$) is  positively oriented in~$\Sigma$. Otherwise,   $\varepsilon(\omega, k)=-1$.  Also,  we let $k\omega$ be the gate orientation   obtained from~$\omega$ by inverting the direction of~$ \alpha_k$ while  keeping  the directions of the other gates.   We
let   $\overline \omega$ denote   the  gate orientation of~$X$   opposite to~$\omega$ on all  gates.

\subsection{Generic loops}\label{koko}\label{kokoff}\label{koko89}\label{moves}    
In the rest of the paper by a loop in~$X$ we mean a circular loop, i.e.,   a continuous map  $a:S^1\to X$. 
The intersection of the set $a(S^1)$ with the $k$-th gate $\alpha_k\subset X$ is denoted   $a\cap \alpha_k$.  A \emph{generic loop}~$a$ in~$X$ is  a   loop   in~$X$ such that (i)  all  branches of~$a$   in~$\Sigma'$ are smooth immersions meeting $\partial \Sigma'$ transversely   at a finite set of points lying  in the  interior  of the gates, and (ii)  all self-intersections of~$a$ in $\Sigma'$ are double transversal intersections lying in $\Int(\Sigma')=\Sigma' \setminus \partial \Sigma'$. The set of self-intersections in $\Sigma'$  (= double points) of a generic loop~$a$   is denoted by   $\# a $. This set is finite and  lies in $\Int(\Sigma')$.

A generic loop~$a$ in~$X$  never traverses a point of a gate  $\alpha_k$   more than once, and the set $a\cap \alpha_k$ is finite.    The \emph{sign} $\varepsilon_p(a)$ of~$a$ at a  point $p\in a\cap \alpha_k$ is    $+1$ if~$a$ goes near~$p$ from $X\setminus \Sigma'$ to $\Int (\Sigma')$ and $ -1$ otherwise.

We define six  local  moves $L_0 - L_5$  on  a  generic loop~$a$  in~$X$  keeping its  free homotopy class. The move $L_0$ is a deformation of~$a$  in the class of generic loops. This move preserves the number $\card(\# a)$.
  The   moves $L_1 -  L_3$   modify~$a$  in a small disk in $\Int(\Sigma')$ and  are modeled on the   Reidemeister moves on knot diagrams (with   over/under-data dropped).   The   move $L_1$ adds a small curl to~$a$ and increases $\card(\# a)$  by~$1$. The  move $L_2$ pushes a branch of~$a$ across another branch of~$a$ increasing $\card(\# a)$  by~$2$.  The  move $L_3$ pushes a branch of~$a$  across a double point  of~$a$ keeping $\card(\# a)$. The  move $L_4$ pushes a branch of~$a$ across  a gate   keeping  $\card(\# a)$. The  move  $ L_5$  pushes a double point of~$a$ across a gate decreasing   $\card(\# a)$  by  $ 1$.  Graphically, the moves $L_4, L_5$ are   similar  to $ L_2, L_3$. 
We call the moves $L_0 - L_5$ and their inverses   \emph{loop moves}.  It is clear  that   generic loops in~$X$ are freely homotopic  if and only if  they can be related  by a finite sequence of loop moves.

   A finite family of   loops in~$X$    is  \emph{generic}  if all  these loops  are generic and all  their mutual crossings   in~$\Sigma'$ are double transversal intersections in $\Int(\Sigma')$. In particular,   these loops  can  not meet at the gates. 
We will use the following notation.  For  generic   loops $a,b$ in~$X$, consider the  set of triples   
$$T(a,b)=\{(k,p,q) \, \vert \, k\in \pi_0(\alpha), p\in a\cap \alpha_k, q\in b\cap \alpha_k, p\neq q \}.$$
Given a gate orientation~$\omega$ of~$X$, we define a  set $T_\omega (a,b )\subset T(a,b)$ by  
$$T_\omega (a,b)=\{(k,p,q) \in T(a,b) \, \vert \, q<_\omega p  \}.$$ Clearly, 
$$ T(a,b) \setminus  T_\omega (a,b)  =\{(k,p,q) \in T(a,b) \, \vert \, p <_\omega q  \}= T_{\overline \omega} (a,b) .$$

  %3. In generalization of quasi-surfaces, one can define a quasi-manifold of any dimension $m\geq 1$ as the result of a gluing of an $m$-dimensional manifold~$M$ to a  topological space along several disjoint  balls in $\partial M$. 

\section{Homological  intersection  forms}\label{section2}

As a prelude to more sophisticated  operations, we define   here   intersection forms in 1-homology of~$X$.

% Notions and techniques introduced here will be used in the sequel. 

% and  $T(a,b)$ is a disjoint union of  $ T_\omega (a,b) $ and $T_{\overline \omega} (a,b)$.  

 %   In the sequel,  we suppose   that~$X$ is orientable  and  let~$\omega$  be    an orientation of~$X$.

\subsection {First homological intersection  form}\label{fofo}  
Given a gate orientation~$\omega$ of~$X$, we   define   a   bilinear form 
\begin{equation}\label{action}
\cdot_\omega: H_1(X;R)\times H_1(X;R) \to R
\end{equation}    called the \emph{first homological intersection  form of~$X$}.   The idea  is to properly  position the loops    near the gates and then  to count  intersections of the  loops in the surface part $\Sigma' \subset X$ of~$X$  with   signs.
We say that an (ordered)  pair  of   loops $a,b$   in~$X$  is     \emph{$\omega$-admissible}  if this pair  is generic   and $T_\omega (a,b)=\emptyset$ so that the crossings of~$a$ with every gate  lie on the $\omega$-left of the  crossings of~$b$ with this gate.   Taking  a   generic pair of  loops $a,b$  in~$X$ and pushing the branches of~$a$  crossing the gates to the $\omega$-left  and  pushing the branches of~$b$ crossing the gates to the $\omega$-right, we obtain an  $\omega$-admissible pair of loops (possibly, with more crossings than the initial pair).
Thus,  any pair of loops in~$X$ may be deformed into an $\omega$-admissible pair. 

For a generic pair   of loops $a,b$ in~$X$,  the set    of   crossings of~$a$ with~$b$ in~$\Sigma'$  is  a  finite subset of  $ \Int(\Sigma')$ denoted   $a\cap b$.  For a  point   $r \in a\cap b$,   set $\varepsilon_r  (a,b) =  1$ if  the (positive) tangent vectors of~$a$ and~$b$ at~$r$ form an   $\omega$-positive basis in the tangent space of~$\Sigma'$ at~$r$ and  set  $\varepsilon_r   (a,b)=-1$ otherwise.

 %The \emph{crossing number} of~$a$ and~$b$ is the integer  $ a \cdot_\omega b   =\sum_{r\in a \cap b} %\varepsilon_r  (a,b)$. We emphasize that the intersections of the loops $a,b$   in  $  X\setminus \Sigma'$ do not %contribute to $  a \cdot_\omega b  $.

\begin{lemma}\label{1alele}   For any $\omega$-admissible pair $a,b$ of   loops in~$X$,   the \lq\lq crossing number"  \begin{equation}\label{actionqq} a \cdot_\omega b =  \sum_{r\in a\cap b} \varepsilon_r  (a,b) \in R \end{equation} depends only on  the homology classes of $a,b$ in  $ H_1(X; R)$. The formula $(a,b) \mapsto  a \cdot_\omega b$ defines a bilinear form \eqref{action}. 
\end{lemma}
\begin{proof} For each $k\in \pi_0(\alpha)$,    one  endpoint  of  the gate $  \alpha_k$ lies on the $\omega$-left of the other  endpoint. Pick disjoint closed  segments  $\alpha_k^- \subset \alpha_k $ and $\alpha_k^+\subset \alpha_k $     containing  these  two endpoints respectively.  Clearly,   $p<_\omega q$ for all $p\in \alpha_k^- $ and $q\in \alpha_k^+$. 
%We call   $\alpha^-_k  $ and $\alpha^+_k  $ respectively the \emph{$\omega$-left  pass}  and   the %\emph{$\omega$-right pass}    of    the gate $ \alpha_k$.  
We  say that a loop  in~$X$ is \emph{$\omega$-left} (respectively, \emph{$\omega$-right}) if it is generic and meets the gates of~$X$ only  at points of $\cup_k \alpha^-_k$ (respectively, of $\cup_k \alpha^+_k$). Given an $\omega$-admissible pair of loops $a,b$ in~$X$, we can  push the branches of~$a$ crossing the gates to the left and  push the branches of~$b$ crossing the gates to the right without creating or destroying   intersections between~$a$ and~$b$. Consequently,~$a$ is homotopic (in fact, isotopic) to an $\omega$-left loop $a'$ and~$b$ is homotopic to an $\omega$-right loop $b'$ such that  $a \cdot_\omega b=a' \cdot_\omega b'$. Since $\alpha^-_k $ is a  deformation retract  of   $\alpha_k  $ for all~$k $, any $\omega$-left loops     homotopic    in~$X$ are homotopic in the class of $\omega$-left loops in~$X$. Similarly, any   $\omega$-right loops     homotopic   in~$X$ are homotopic in the class of $\omega$-right loops in~$X$. Such homotopies  of $a', b'$ obviously preserve   $a' \cdot_\omega b'$. Therefore  the  number $a \cdot_\omega b=a' \cdot_\omega b'$ depends only on the (free)  homotopy classes of  $a,b$ in~$X$. Moreover, since this   number   linearly depends on both loops~$a$ and~$b$,  it depends only on their homology classes. This implies the  claim of the lemma.
\end{proof}

%Both~$a$ and~$b$ are products of a finite number of  subpaths running   in $\Sigma'$ or  in the closure of $X\setminus \Sigma'   $.

We emphasize that the crossings of   loops  in  $  X\setminus \Sigma'$ do not contribute to the crossing number. For loops   in the surface $\Sigma' \subset X$, the crossing number is the usual homological intersection  number. The crossing numbers of  loops  in $X\setminus \Sigma'$ with arbitrary loops in~$X$ are equal to zero.

For an  $\omega$-admissible pair $a,b$ of loops in~$X$, the   pair $b,a$ is $\overline\omega$-admissible. 
 Using these pairs to compute   $ a \cdot_\omega b  $ and $  b  \cdot_{\ \overline \omega} a$, we obtain two sums which differ only in the signs of the  terms. 
Hence, for any $x,y \in H_1(X;R)$,   \begin{equation}\label{chorisdf}    x \cdot_\omega y    = - y \cdot_{\ \overline \omega} x    .\end{equation}

% The next lemma calculates the crossing number for all generic   pairs   of loops. 

\begin{lemma}\label{1aeee}    For any homology classes $x,y \in H_1(X;R)$ represented by a generic pair of loops $a,b$ in~$X$,
\begin{equation}\label{chorieemvmmve}
 x \cdot_{\omega}  y   =  \sum_{r\in a\cap b } \varepsilon_r   (a,b) +\sum_{(k,p,q) \in T_{\omega} (a,b)}\,    \varepsilon(\omega, k)   \,\varepsilon_p(a)  \,\varepsilon_q(b). \end{equation}
\end{lemma}
\begin{proof}  Consider an $\omega$-admissible pair of  loops $a',b'$  obtained from $a,b$ by pushing the branches of~$a$ meeting the gates to the $\omega$-left of the branches of~$b$ meeting the   gates.  This   modifies~$a$ in a small neighborhood of the gates; we  can assume that   $a', b'$ have the same intersections in $\Sigma'$ as $a,b$ plus one  additional intersection  $r=r(k,p,q) \in \Sigma'$   for each  triple $(k,p,q)\in T_\omega (a,b)$. Observe that  $\varepsilon_r  (a', b')= \varepsilon(\omega, k) \, \varepsilon_p(a)\,  \varepsilon_q(b)$.  Consequently,    $$
x \cdot_{\omega}  y  =a' \cdot_\omega b'=\sum_{r\in a'\cap b'} \varepsilon_r  (a',b') $$ 
$$=\sum_{r\in a\cap b} \varepsilon_r  (a,b)+\sum_{(k, p,q) \in T_{\omega} (a,b)} \varepsilon(\omega, k) \, \varepsilon_p(a)\,  \varepsilon_q(b).$$
\end{proof}

 Formula \eqref{chorieemvmmve} generalizes   \eqref{actionqq} because $T_\omega (a,b)=\emptyset$   for an  $\omega$-admissible pair  of loops $a,b$. We  describe next the dependence of $\cdot_\omega$ on~$\omega$. We will  use the linear map $v_k: H_1(X;R)\to R$ \lq\lq dual'' to the  $k$-th gate. This map carries the homology class of any generic loop~$a$ to $\sum_{p\in a \cap \alpha_k} \varepsilon_p(a)$.  
 In the notation of Section~\ref{The dual map}, $v_k=v_{\alpha_k}$.

 % \lq\lq dual" to the  $k$-th gate~$\alpha_k$. Let $S^1\subset \CC$ be the unit circle with  counterclockwise orientation, and  let $v\in H^1(S^1;R)=R$ be the   generator determined by this orientation.  Let $\alpha \times [0,2]  \to \Sigma$ be a neighborhood of $\alpha$ in~$\Sigma$ extending the neighborhood   $\alpha \times [0,1]  \subset \Sigma$ considered  in Section~\ref{notat}.  Set   $v_k=f_k^* (v)\in   H^1(X;R)$ where  the map $f_k:X\to S^1$   
  %carries each  point $(x,t)\in \alpha_k\times (0,2)\subset X$ to $exp(\pi i t)\in S^1$ and carries the rest of~$X$ to $1\in S^1$.  The  This  easily implies  that  $\sum_{k\in \pi_0(\alpha)} v_k=0$. 

\begin{theor}\label{1a}    For any $x,y\in H_1(X;R)$ and   $k_0\in \pi_0(\alpha)$,
\begin{equation}\label{chori}
  x \cdot_{k_0\omega}  y   =   x \cdot_{\omega} y   - \varepsilon(\omega, k_0) \,  v_{k_0}(x) \,  v_{k_0}(y). \end{equation}
%As a consequence, for any  set  $K \subset \pi_0(\alpha)$,
%\begin{equation}\label{choriupso}
  %x \cdot_{K\omega} y     =x \cdot_{\omega} y - \sum_{k\in K} \varepsilon(\omega, k) \,  v_k(x) \, v_k(y). %\end{equation}
\end{theor}
\begin{proof}  Pick an $\omega$-admissible  pair of loops $a,b$ representing respectively $x,y$. We  compute $x  \cdot_{\omega} y=a \cdot_{\omega} b$ from the definition and compute $x \cdot_{k_0\omega} y= a \cdot_{k_0\omega} b$ from  Lemma~\ref{1aeee}.   The resulting   expressions   differ   in the sum associated with    $ T_{k_0\omega} (a,b)$. Since the pair $(a,b)$ is $\omega$-admissible, the set $T_{k_0\omega} (a,b) $ consists of  all triples $(k_0, p\in a \cap \alpha_{k_0} ,q\in b \cap \alpha_{k_0})$. Therefore
$$x \cdot_{k_0\omega} y    =x \cdot_{\omega} y + \sum_{p\in a \cap \alpha_{k_0} ,q\in b \cap \alpha_{k_0}}\, \varepsilon(k_0\omega, k_0) \,  \varepsilon_p(a) \,  \varepsilon_q(b)$$
$$= x \cdot_{\omega} y - \varepsilon(\omega, k_0) \,  v_{k_0}(x) \,  v_{k_0}(y).$$
\end{proof}

%
%\begin{corol}\label{CQPSa}    For any $x,y\in H_1(X;R)$, \begin{equation}\label{chori++es}
%x \cdot_{\omega} y +  y \cdot_{\omega} x =\sum_{k\in \pi_0(\alpha)}  \varepsilon(\omega, k) \,  v_k(x) \, v_k(y). \end{equation}
%\end{corol}
%\begin{proof} 
%Applying Formula~\eqref{chori}  consequtively to all elements of $\pi_0(\alpha)$, we get
%$$x \cdot_{\overline \omega} y=x \cdot_\omega y-\sum_{k\in \pi_0(\alpha)}  \varepsilon(\omega, k) \,  v_k(x) \, v_k(y). $$
%Permuting here $x$, $y$, and using that $y\cdot_{\overline \omega}  x= - x \cdot_{ \omega}  y$,    we obtain  \eqref{chori++es}.
%\end{proof}

\subsection {Second homological intersection  forms}\label{foddfo}  
Pick a  gate orientation~$\omega$ of~$X$ and define a  skew-symmetric bilinear form  $i_X: H_1(X;R) \times H_1(X;R) \to R$
  by 
\begin{equation}\label{chori+nmnm+es} i_X (x,y)= x \cdot_{\omega} y - y\cdot_{\omega} x  \end{equation} 
for   $x,y \in H_1(X;R)$. This form   does not depend on~$\omega$ because, by~\eqref{chori}, 
$$x \cdot_{k\omega} y  -  y \cdot_{k\omega} x   =x \cdot_{\omega} y - y\cdot_{\omega} x $$ 
  for any $x,y\in H_1(X;R)$ and    $k \in  \pi_0(\alpha)$.  
 We call   $i_X$ the \emph{second homological intersection  form of~$X$}. Both the first and the second  homological intersection forms   generalize  the standard   intersection form  in  1-homology of   a surface. Indeed, the value of  the form \eqref{action} (respectively,  \eqref{chori+nmnm+es})  on any pair of  homology classes  of loops in $\Sigma'\subset X$    is equal to    the usual   intersection number  of these loops in~$\Sigma'$ (respectively,  twice this number).

\begin{theor}\label{CQPeeeSa}  For any gate orientation~$\omega$   and any $x,y\in H_1(X;R)$, we have \begin{equation}\label{cnnnhori++es}
2 x \cdot_{\omega} y   = i_X(x,y) +\sum_{k\in \pi_0(\alpha)}  \varepsilon(\omega, k) \,  v_k(x) \, v_k(y) \end{equation}
\end{theor}
\begin{proof} 
Applying~\eqref{chori}  consequtively to all elements of $\pi_0(\alpha)$, we get
$$x \cdot_{\overline \omega} y=x \cdot_\omega y-\sum_{k\in \pi_0(\alpha)}  \varepsilon(\omega, k) \,  v_k(x) \, v_k(y). $$
Substituting  $x\cdot_{\overline \omega}  y= - y \cdot_{ \omega}  x$,    we get
$$
x \cdot_{\omega} y +  y \cdot_{\omega} x =\sum_{k\in \pi_0(\alpha)}  \varepsilon(\omega, k) \,  v_k(x) \, v_k(y). $$
This formula and the equality $   x \cdot_{\omega} y - y\cdot_{\omega} x = i_X (x,y) $ imply  \eqref{cnnnhori++es}.
\end{proof}
 
Formula~\eqref{cnnnhori++es} shows that if $1/2\in R$, then $\cdot_\omega$ is a sum of  $(1/2) i_X$ and   terms associated with  the gates.

\subsection{Remark}\label{modulo2}   For $R=\ZZ/2\ZZ$, the definitions in this section and below  do not depend on the orientation of~$\Sigma$ and  extend  to non-orientable quasi-surfaces.

 \section{Homotopy intersection forms}\label{section3}

We define  homotopy intersection forms of~$X$ refining the homological   forms above.
In this section and below,    ${\mathcal L} = {\mathcal L} (X)$   is the set of free   homotopy classes of loops in~$X$  
 and   $ R{\mathcal L}$ is the free $R$-module with basis~${\mathcal L}$.
By Sections~\ref{Derivative brackets in group algebras} and~\ref{A topological example},   for each $m\geq 1$, the   gate $\alpha_k\subset X$   determines  a cyclically symmetric  $m$-bracket in  $R\mathcal L$. It is  
denoted $\mu^{m}_{k}
$.

%  In this section, ${\mathcal L}={\mathcal L}(X)$  is   the set of free   homotopy classes of loops in~$X$ and  $ R{\mathcal L}$  is the  free module with basis~${\mathcal L}$. 

\subsection{First homotopy intersection form}\label{The form bullet}   
Pick a gate orientation~$\omega$ of~$X$. 
Any  pair    $x,y \in {\mathcal L} $ can be represented by an  $\omega$-admissible pair of loops $a,b$ in~$X$, cf. Section~\ref{fofo}. For  a point $r\in a\cap b$,  consider the loops  $a_r, b_r$  which are reparametrizations of $a, b$, respectively,     starting and ending  in~$r$.   Consider the product loop $a_r b_r$ based in~$r$ and  set
\begin{equation}\label{deff} x \bullet_\omega y=\sum_{r\in a\cap b} \varepsilon_r   (a,b) \langle a_r b_r \rangle \in  R {\mathcal L}  \end{equation} where for a    loop~$c$ in~$X$, we let  $\langle c \rangle \in {\mathcal L} $ be its free homotopy class.  The sum on the right-hand side of \eqref{deff} is an algebraic sum of all possible ways to graft~$a$ and~$b$. This sum
  is preserved under   all loop moves on $a,b$ keeping  this pair   $\omega$-admissible. Hence,  $x \bullet_\omega y$ does not depend on the choice of $a,b$ in the   homotopy classes  $x,y$. Extending the map  $(x,y) \mapsto x \bullet_\omega y$ by bilinearity, we obtain a  bilinear pairing \begin{equation}
\label{ghjk} \bullet_\omega:R {\mathcal L}   \times R {\mathcal L}  \to R {\mathcal L}  . \end{equation}  We call  this pairing  the  \emph{first homotopy intersection form of~$X$}. 
The proof of Formula~\eqref{chorisdf} applies here and  shows that  for any $x,y\in R {\mathcal L} $, 
\begin{equation}\label{chorisdfbulle}   
 x \bullet_\omega y =- y \bullet_{\overline \omega} x.   \end{equation}

For a generic (non-double) point~$p$ of a generic loop~$a$, we let $a_p$ be the   loop  which starts at~$p$ and goes along~$a$ until coming back to~$p$. Having two generic loops $a,b$ and  points $p  \in a\cap \alpha_k$, $q\in q\cap \alpha_k$ on the same gate,  we can   multiply the loops $a_p, b_q$  using an arbitrary path in $\alpha_k$   connecting their base points $p,q$. The resulting loop determines a well-defined element of $\mathcal L$   denoted $\langle a_p b_q \rangle$. 

\begin{lemma}\label{1aeee++}   Let $x,y \in {\mathcal L}$  be represented by a generic pair of loops $a,b$. Then 
\begin{equation}\label{chorieee++}
 x \bullet_{\omega} y   =  \sum_{r\in a\cap b} \varepsilon_r  (a,b) \langle a_r b_r \rangle +\sum_{(k,p,q)\in T_\omega (a,b) }\,    \varepsilon(\omega, k)   \,\varepsilon_p(a)  \,\varepsilon_q(b) \langle a_pb_q \rangle. \end{equation}
\end{lemma}

The proof repeats the proof of Lemma~\ref{1aeee} with   obvious modifications.  If $a\cap b=\emptyset$, then  \eqref{chorieee++} simplifies to   
\begin{equation}\label{chorieee++vj}
 x \bullet_{\omega} y   =  \sum_{(k,p,q) \in T_\omega (a,b) }\,    \varepsilon(\omega, k)   \,\varepsilon_p(a)  \,\varepsilon_q(b) \langle a_pb_q \rangle. \end{equation}
 
%%For an $\omega$-admissible pair $a,b$, Formula \eqref{chorieee++} coincides with \eqref{deff} as in this case $T_\omega (a,b)= \emptyset$. 

\begin{theor}\label{1abulle}    For any $x,y\in R {\mathcal L}$ and     $k\in \pi_0(\alpha)$,
\begin{equation}\label{cvhoribulle}
  x \bullet_{k\omega}  y   =   x \bullet_{\omega} y   - \varepsilon(\omega, k) \,  \mu^{2}_{k}(x,y)  . \end{equation}
\end{theor}
\begin{proof}  It suffices to handle the case   $x,y \in {\mathcal L}  $.  Then   proceed as in the proof of Corollary~\ref{1a} replacing $\cdot$ by $\bullet$ and using Formula~\eqref{twotwo3aa} to compute $\mu^{2}_{k}=\mu^{2}_{\alpha_k}$.  
\end{proof}

Applying \eqref{cvhoribulle} consecutively  to all $k\in \pi_0(\alpha)$ and using~\eqref{chorisdfbulle}, we get 
 
\begin{corol}\label{CQPSabulle}  
 For any $x,y\in R {\mathcal L}$, \begin{equation}\label{chori++}
x \bullet_{\omega} y +  y \bullet_{\omega} x =\sum_{k\in \pi_0(\alpha)}  \varepsilon(\omega, k) \,  \mu^{2}_{k}(x,y). \end{equation}
\end{corol}

  \subsection{Second homotopy intersection form}\label{The dbdbdform bullet}   We  define a   2-bracket $[-,-] $ in $ R \mathcal L$
  by 
$[x,y]=x \bullet_\omega y- y \bullet_\omega x $ for all $x,y \in R {\mathcal L}$.   This   skew-symmetric bracket    does not depend on~$\omega$ because, by~\eqref{cvhoribulle}, 
$$x \bullet_{k\omega} y  -  y \bullet_{k\omega} x   =x \bullet_{\omega} y - y\bullet_{\omega} x $$ 
  for all $k \in  \pi_0(\alpha)$.  (Here we use the symmetry of the brackets $\{\mu^2_k\}_k$.)  
 We call the 2-bracket   $[-,-]$ the \emph{second homotopy intersection  form of~$X$}. Both   the first and the second homotopy intersection forms     generalize Goldman's~\cite{Go1}, \cite{Go2}   bracket: 
 the value of   $\bullet_\omega$ (respectively,  $[-,-] $)  on any pair of  free homotopy  classes  of loops in $\Sigma'\subset X$    is equal to  their  Goldman's    bracket    (respectively,  twice this bracket).   
     
 Theorem~\ref{1aeee++} allows us to compute $[x,y] $ for   $x,y \in {\mathcal L}$   from any generic pair of loops $a,b$ representing  $x, y$ and any gate orientation~$\omega$ of~$X$. Namely, 
\begin{equation}\label{chorieee++veryryr}
[x,y]     \end{equation}
$$ = 2 \sum_{r\in a\cap b} \varepsilon_r  (a,b) \langle a_r b_r \rangle +\sum_{(k,p,q) \in T(a,b) }\,   \delta_\omega (p,q)\,  \varepsilon(\omega, k)   \,\varepsilon_p(a)  \,\varepsilon_q(b) \langle a_pb_q \rangle$$
where $ \delta_\omega (p,q)=1$ if   $p>_\omega q$ and 
$ \delta_\omega (p,q)=-1$ if  $p<_\omega q$.    Note also the identity
$$
2x \bullet_{\omega} y = [x,y] +\sum_{k\in \pi_0(\alpha)}  \varepsilon(\omega, k) \,  \mu^{2}_{k}(x,y)  $$which  can be easily deduced from~\eqref{chori++}.  
Consequently,  if $1/2\in R$, then the form $\bullet_\omega$ expands as a sum of $(1/2) [-,-]$  and  terms associated with the gates.

\subsection{Remark}  Other algebraic operations  associated with  surfaces  may be extended to quasi-surfaces. This includes algebraic intersections of loops (see \cite{Tu1}), Lie cobrackets 
(see \cite{Tu2}, \cite{Hain}), double brackets   and generalized Dehn twists (see \cite{MTold}), and  quasi-Poisson structures on the representation spaces (see \cite{MTnew}). In a sequel to this paper, the author plans to discuss  natural cobrackets 
appearing in the study of quasi-surfaces.

 \section{Main theorem}\label{section3ghgh}

\subsection{Statement}  We state   our  main  result concerning the quasi-surface~$X$. 
  Let  $\pi=\pi_1(X, \ast)$ with $\ast\in Y \subset X$. Consider the group algebra 
$A=R[\pi]$ and  the second homotopy intersection form $[-,-]: \check A \times \check A \to \check A$  in $ \check A=R \mathcal L$.  The next theorem computes the Jacobiator of this 2-form via  the 3-bracket
 $$\mu=\sum_{k \in \pi_0(\alpha)} \mu^{3}_{k}:  (\check A)^3 \to \check A.$$ This theorem shows  that the failure of the intersection form  to satisfy the Jacobi identity is entirely due to the presence of the gates.

\begin{theor}\label{structe}   The brackets  $[-,-]$ and $ \mu$ are braces in~$  A$ forming  a quasi-Lie pair.
\end{theor}

This theorem can be rephrased  by saying that the pair $[-,-], \mu$ turns~$A$ into a brace algebra. 
Combining   Theorems \ref{structe}  and \ref{Cbtheor}, we conclude that for all $n \geq 1$, the $n$-th trace algebra  $A^t_n$ of $A$     carries a unique structure of a brace algebra  
 such   that 
  the trace ${\text {tr}}:  \check A \to A_n^t$  is a brace    homomorphism.     
  
 By  Example~\ref{Exampleskk}.1 (case $m=1$), every surface~$\Sigma$ is a quasi-surface with a single gate  which separates the surface part (a copy of~$\Sigma$) from a cone over a segment in $\partial \Sigma$. Here  $\mu=0$ since  each  loop in this quasi-surface may be deformed into the complement of the gate.  Theorem~\ref{structe}  yields then  the usual Jacobi relation for $[-,-]$. On the other hand,  in Example~\ref{Exampleskk}.4 it may well happen that $\mu\neq 0$.

  The proof of Theorem~\ref{structe} occupies the rest of the section.

\subsection{Proof of Theorem~\ref{structe}: beginning}   Throughout the proof we fix a gate orientation~$\omega$  of~$X$. 
By Section~\ref{Gate brackets},  the brackets $\{\mu^{3}_{k}\}_k$  in $ \check A$ are  cyclically symmetric    braces in~$A$ (independent of~$\omega$). Therefore so is their sum~$\mu$. The skew-symmetry of  the  2-bracket  $[-,-] $ in  $ \check A$  is obvious.   We now prove  that this  2-bracket    is a brace in~$A$.
Since it  is skew-symmetric, it suffices  to prove that $[-,-] $   is a  weak derivation in the second  variable.    Pick any  $x  \in \mathcal L =\check \pi$,  $y\in \pi$ and   represent the pair $x,y$ by an  $\omega$-admissible pair of loops $a,b$ in~$X$ where~$b$ is  based in~$\ast$.  For each point $r\in a\cap b$,  consider the loop  $a_r$   obtained by reparametrization of~$a$ so that its  starts and ends  in~$r$.  We have $b=b^-_r b^+_r$ where  $b^-_r$ is the path in~$X$   going  from $\ast$ to~$r$ along~$b$ and  $b^+_r$ is the path in~$X$   going  from~$r$ to~$\ast$ along~$b$. The product path $b^-_r  a_r   b^+_r$ in~$X$  is a loop based in~$\ast$;  consider its homotopy class  $[ b^-_r  a_r   b^+_r ] \in \pi $.    Recall  the crossing sign  $ \varepsilon_r  (a,b)=\pm 1 $ 
 and set
\begin{equation}\label{vvdeff} x  \bullet_\omega y=\sum_{r\in a\cap b} \varepsilon_r  (a,b)\, [ b_r^- a_r  b_r^+ ]  \in  A=R [\pi] . \end{equation}  The sum on the  right-hand side is an algebraic sum of all possible ways to graft~$a$ to~$b$.  (For surfaces, the pairing  \eqref{vvdeff}    was first  introduced by  Kawazumi and Kuno  \cite{KK}, \cite{KK2}.) It is easy to see from the definitions  that (i)  the expression $x  \bullet_\omega y$   depends only on $x, y$ and does not depend on the choice of  $a,b$ and (ii) the linear extension  $ A \to  A$ of the map $ y \mapsto  x \bullet_\omega y $   is a derivation of the algebra~$A$. 
  Next, denote the projection $ A\to \check A$  by~$p$ and note that for each   $y\in \pi\subset A$, its image $p(y) \in \check \pi\subset  \check A$ is the conjugacy class   of~$y $ in~$\pi$.   Comparing  Formulas  \eqref{deff} and
\eqref{vvdeff},   we obtain that  $ p(x\bullet_\omega y)  = x \bullet_\omega p(y)$.  By~\eqref{chorisdfbulle}, 
$$[x, p(y)] =  x \bullet_\omega p(y) -  p(y) \bullet_\omega  x
=  x \bullet_\omega p(y) +  x \bullet_{\overline \omega}  p(y)$$ $$=p(x \bullet_\omega y)  + p(  x \bullet_{\overline \omega} y)=p(x \bullet_\omega y  +  x \bullet_{\overline \omega} y).$$
Consequently,  the linear endomorphism  $  [x,  -] $ of $\check A$ is induced by the linear 
 endomorphism  $y \mapsto  x\bullet_\omega  y +x \bullet_{\overline \omega  } y$ of~$A$. Since  the  maps $y \mapsto  x\bullet_\omega y$ and $y \mapsto x \bullet_{\overline \omega  } y$ are derivations of~$A$,  so is their sum $y \mapsto  x \bullet_\omega y  +  x \bullet_{\overline \omega} y$.  This implies that  the bracket $[-,-]$ is   a  weak derivation in the second  variable and  is a brace.

It remains to verify   the  Jacobi-type identity  \eqref{Jaco}. This is done in the next three subsections.

 \subsection{Preliminaries on simple loops}\label{section3loops}  We  say that a finite family  of loops in the quasi-surface ~$X$  is \emph{simple} if   these loops  meet the gates of~$X$ transversely and have no  mutual  crossings or self-crossings   in the surface part~$\Sigma'\subset X$ of~$X$.  A simple family of loops is generic.

 \begin{lemma}\label{Loopsiemnn}  Any finite family of loops in~$X$  can  be deformed in~$X$ into a simple family of loops.
\end{lemma}

\begin{proof} Consider first a   single loop in~$X$. Since~$X$ is path-connected and contains a gate, we can deform our  loop  into a  generic loop~$a$ which  meets a gate at least once.
 If the set $\# a$ of double points of~$a$ in $\Int(\Sigma')$ is empty, then we are done. 
 Otherwise, pick a  point $r \in \# a$. Starting   at $r=r_0$ and moving along~$a$ (in the given direction of~$a$), we     meet several double  points $r_1,..., r_n  \in \# a$ with $n\geq 0$ and then  come to a  point $p\in a\cap \alpha_k$ of a certain   gate $\alpha_k$.   The  segment~$b$  of~$a$   connecting  $r_n$ to~$p$  is  embedded    in $\Sigma'$  and meets  $  \# a$ only at its endpoint $r_n$. Let~$c$  be   the branch of~$a$ transversal to~$b$ at   $r_n$.   Push the branch~$c$ towards~$p$ along~$b$ while  keeping~$c$ and~$b$ transversal  and eventually push~$c$ across  $\alpha_k$ at~$p$. This  transformation of~$a$ decreases ${\rm {card}} (\# a)$ by 1 and increases    $\card (a \cap \alpha_k)$ by 2. Continuing by induction, we    deform our  loop into a generic loop without   self-intersections  in~$\Sigma'$.  
 If the original  family of loops contains $\geq 2$ loops, then we first deform it into a generic family of loops which  all  meet some gates. Then pushing branches at crossings and self-crossings  as above, we deform the latter family into a simple family of loops. \end{proof} 
 
%(By moving along a loop we   mean moving in the given direction of this loop.)

  \subsection{Preliminaries on sign functions}\label{section3loops+} We define   two  functions  used in the proof.    The first function, $\delta$, is defined on the set $\{\pm 1\}=\{-1, 1\}$ by  $\delta (1)=1 $  and  $\delta (-1)=0 $. The second function, also denoted $\delta$,  is defined on the set of all triples
$  \varepsilon, \varepsilon', \varepsilon''\in \{\pm  1\}$  by  the formula 
\begin{equation}\label{dd} \delta( \varepsilon, \varepsilon', \varepsilon'')=  \varepsilon  \varepsilon' \delta (\varepsilon'') + 
 \varepsilon  \varepsilon'' \delta (\varepsilon') + \varepsilon'  \varepsilon'' (1- \delta (\varepsilon)).\end{equation} This function   is    invariant under all permutations of   
  $  \varepsilon, \varepsilon', \varepsilon''$ as easily follows from the identity $  2\delta (\varepsilon) =\varepsilon+1$ for all $\varepsilon \in \{\pm 1\}$.  The same identity   implies another useful equality:  for all  $  \varepsilon, \varepsilon', \varepsilon''\in \{\pm 1\}$, we have  
  \begin{equation}  \delta( \varepsilon, \varepsilon', \varepsilon'') -\varepsilon  \varepsilon'  \varepsilon''= \label{dd+}   \varepsilon  \varepsilon' \delta (\varepsilon'') + 
 \varepsilon  \varepsilon'' (1-\delta (\varepsilon')) + \varepsilon'  \varepsilon'' (1- \delta (\varepsilon)) . \end{equation}

 \subsection{Proof of  \eqref{Jaco}}\label{section3ghvvgh}    Any (possibly, non-associative) algebra~$\mathcal R$  carries the bracket $[x,y]=xy-yx$. For  $x,y,z\in \mathcal R$,
 set  $$ P(x,y,z)=(xy)z+ (yz)x+(zx)y+z(yx)+x(zy)+y(xz) \in \mathcal R.  $$ A direct computation shows that  \begin{equation}\label{key-} [[x,y], z] +[[y,z],x]+[[z,x],y]=P(x,y,z) - P(y,x,z).  \end{equation}  We apply these observations to the algebra $\check A=R\mathcal L$ with multiplication $\bullet=\bullet_\omega$. In view of \eqref{key-}, the  identity  \eqref{Jaco} is equivalent to the identity    \begin{equation}\label{key}
P(x,y,z) - P(y,x,z)= \mu(x,y,z) - \mu (y,x,z)   \end{equation}  for all $x,y,z\in \check A$.
  Since both sides are linear in $x,y,z$, it suffices to handle  the case   $x,y,z\in {\mathcal L}$. Set 
  \begin{equation}\label{idyvg3n}     u_\omega(x,y,z)= (x \bullet_\omega   y) \bullet_\omega  z  +(y \bullet_\omega   z) \bullet_\omega  x+(z \bullet_\omega   x) \bullet_\omega  y   .  \end{equation}
 Formula~\eqref{chorisdfbulle} implies that
\begin{equation}\label{idyvg3n+} u_{\overline \omega} (x,y,z)=z \bullet_\omega  ( y  \bullet_\omega x)  + x \bullet_\omega  ( z \bullet_\omega  y)+ y \bullet_\omega   (x \bullet_\omega  z)  . \end{equation}
 Thus,    \begin{equation}\label{idvnvnyvg3n+}  P(x,y,z)= u_\omega(x,y,z)+u_{\overline \omega} (x,y,z).  \end{equation} 
  In our computations, we   represent  $x,y,z$  by  loops $a,b,c$ in~$X$, respectively.

   If  the  loops $a,b,c$  lie in $\Sigma'\subset X$ then Goldman's results   imply   that $ u_\omega(x,y,z)=0 $ for all~$\omega$ so that $P(x,y,z)=0$. It is also clear that $\mu(x,y,z)=0$, and  \eqref{key} follows. For completeness, we  check the identity  $ u_\omega(x,y,z)=0$ in this case  (it is also included in the general case treated below).
Deforming if necessary $a,b,c$,  we can assume that  the triple $\{a,b,c\}$    is generic in the sense of  Section~\ref{koko89}. 
  Then   
 $x \bullet_\omega   y $ is computed by~\eqref{deff}.  To compute $(x\bullet_\omega   y) \bullet_\omega   z$,     consider all intersections of the loop~$c$ with the loops $\{a_r b_r\}_{r\in a\cap b}$. At    such an intersection, say~$s$, the loop~$c$ meets either~$a$ or~$b$. Thus,
 $(x\bullet_\omega   y) \bullet_\omega   z=\sigma(a,b,c) +\tau(a,b,c)$ where
$$ \sigma(a,b,c) =\sum_{r\in a\cap b}  \, \sum_{s\in a\cap c}   \, \varepsilon_r  (a,b)   \, \varepsilon_s  (a,c)  \, \langle b\circ_r a \circ_s c \rangle , $$
$$
\tau(a,b,c) =\sum_{r\in a\cap b} \, \sum_{s\in b\cap c}   \, \varepsilon_r   (a,b)  \, \varepsilon_s  (b,c)  \, \langle a\circ_r  b \circ_s c \rangle.$$
Here  $b\circ_r a \circ_s c$ is the loop obtained by grafting  the loops $b_r, s_c$ to~$a$ at the points $r,s$. More precisely,  this  loop goes along~$b$ starting and ending in~$r$, then     along~$a$ from~$ r$ to~$s$, then along~$c$ starting and ending in~$s$, and finally returns along~$a$  to~$r$.   Note that the inclusions  $r\in a\cap b$,  $s\in a\cap c$  ensure that $r\neq s$ so that the loop $b\circ_r a \circ_s c$  is well-defined. The loop $a\circ_r b \circ_s c$ is defined similarly 
grafting  the loops $a_r, s_c$ to~$b$ at the points $r,s$. Therefore 
$$u_\omega(x,y,z)=\sigma(a,b,c) +\tau(a,b,c)+\sigma(b,c,a) +\tau(b,c,a)+\sigma(c,a,b) +\tau(c,a,b).$$
Note  that $\sigma(a,b,c)   =- \tau(c,a,b)$ as directly follows from the definitions and the identity $\varepsilon_s  (a,c) =-\varepsilon_s  (c,a) $ for  $s\in  a\cap c$. 
Permuting $a,b,c$, we get  $\sigma( b,c,a)   =- \tau(a,b,c) $ and $\sigma(c,a,b)   =- \tau(b,c,a)$.  Summing up, we obtain  $u_\omega(x,y,z)=0$.

Consider now  the general case   where the loops $a,b,c$  do not necessarily lie in~$\Sigma'$.  
By Lemma~\ref{Loopsiemnn},  deforming  $a,b,c$ in~$X$, we can ensure that this triple of loops  is simple. By~\eqref{chorieee++vj},  
$$
 x \bullet_{\omega} y   =  \sum_{(k,p,q) \in  T_\omega (a,b) }\,    \varepsilon(\omega, k)   \,\varepsilon_p(a)  \,\varepsilon_q(b) \langle a_p \gamma_{p,q} b_q \gamma_{p,q}^{-1}  \rangle $$
 where   $a_p, b_q $ are   loops reparametrizing $a,b$ and based respectively in $p, q$  while~$\gamma_{p,q}$ is a path  connecting~$p$ and~$q$  in   $\alpha_k$. We deform the loop $a_p \gamma_{p,q} b_q \gamma_{p,q}^{-1}$    by  slightly pushing its subpaths   $ \gamma_{p,q}^{\pm 1}$   \lq\lq behind the gate'', i.e., into $ X\setminus   \Sigma'$.  (The endpoints $p,q$ of these subpaths  are pushed  into $ X\setminus   \Sigma'$ along $a,b$, respectively.) The resulting  loop is denoted   by $a\circ_{p,q}b$. Thus,   \begin{equation}\label{monoee} 
 x \bullet_{\omega} y   =  \sum_{(k,p,q) \in T_\omega (a,b) }\,   \varepsilon(\omega, k)   \,\varepsilon_p(a)  \,\varepsilon_q(b) \langle a\circ_{p,q}b \rangle. \end{equation}
 Note that the loop $a\circ_{p,q}b$  is simple;   moreover,  the pair formed by this loop and~$c$   is simple.
Applying~\eqref{chorieee++vj} again, we get 
\begin{equation}\label{mono} 
\langle a\circ_{p,q}b \rangle  \bullet_{\omega} z   =  \sum_{(l,s,t) \in  T_\omega (a\circ_{p,q}b, c) }\,    \varepsilon(\omega, l)   
\,\varepsilon_s(a\circ_{p,q}b)  \,\varepsilon_t(c) \langle (a\circ_{p,q}b)_s c_t \rangle. \end{equation}
For any $l\in \pi_0(\alpha)$,   the set     $(a\circ_{p,q} b)\cap  \alpha_l$ is a disjoint  union of   the sets $a\cap \alpha_l$ and $b\cap \alpha_l$. 
   Therefore  $T_\omega (a\circ_{p,q}b, c) =T_\omega (a , c)  \coprod T_\omega ( b, c) $. By \eqref{mono},  \begin{equation}\label{monoef} \langle a\circ_{p,q}b \rangle  \bullet_{\omega} z   =\varphi_\omega(k,p,q) +\psi_\omega (k,p,q)\end{equation} where 
$$\varphi_\omega(k,p,q)=  \sum_{(l,s,t)\in T_\omega (a,c) }\,    \varepsilon(\omega, l)   
\,\varepsilon_s(a)  \,\varepsilon_t(c) \langle (a\circ_{p,q}b)_s c_t \rangle$$
   and
$$\psi_\omega (k,p,q)=  \sum_{(l,s,t) \in T_\omega (b,c) }\,    \varepsilon(\omega, l)   
\,\varepsilon_s(b)  \,\varepsilon_t(c) \langle (a\circ_{p,q}b)_s c_t \rangle.$$
Combining   \eqref{monoee} with  \eqref{monoef}, we obtain  
\begin{equation}\label{monoeg} (x \bullet_{\omega} y )\bullet_\omega z= \sum_{(k,p,q) \in T_\omega (a,b) }\,   \varepsilon(\omega, k)   \,\varepsilon_p(a)  \,\varepsilon_q(b)\big (\varphi_\omega(k,p,q) +\psi_\omega(k,p,q) \big ). \end{equation}

To compute the latter sum, we rewrite  $\varphi_\omega(k,p,q)$ as follows.
For $s\in a \cap \alpha_l$,  the homotopy class   $\langle (a\circ_{p,q}b)_s c_t \rangle$ is represented by the loop $b\circ_{p,q} a  \circ_{s,t} c$ obtained by grafting~$b$ and~$c$ to~$a$  via  a  path  in $\alpha_k$ from $p\in a\cap \alpha_k$ to $ q\in b\cap \alpha_k$ and a path  in $\alpha_l$  from $s\in a\cap \alpha_l$ to $t\in c\cap \alpha_l$. To give a   precise description of this loop,  we  separate two cases. 

Case 1: $p\neq s$ so that  the points $p ,q ,s ,t $ are pairwise distinct (possibly, $k=l$). In this case the loop $b\circ_{p,q} a  \circ_{s,t} c$  starts at~$p$ and goes:   along the gate $\alpha_k$   to~$q$, then along the full loop~$b$ back to~$q$,  then  along $\alpha_k$ back to~$p$, then  along~$a$ to~$s$, then  along the gate  $\alpha_l$ to~$t$, then along the full  loop~$c$ back to~$t$, then along~$\alpha_l$ back  to~$s$, and finally  along~$a$ back to~$p$. 

Case 2: $p= s$. Then  $k=l$ and    $p ,q ,t $ are three distinct points on  the gate  $\alpha_k$.
If $\varepsilon_p(a)=+1$, then the loop $b\circ_{p,q} a  \circ_{s,t} c$  starts at~$p$ and goes:   along $\alpha_k$   to~$q$, then along the full loop~$b$ back to~$q$,  then  along $\alpha_k$ to~$t$, then along the full  loop~$c$ back to~$t$, then along $\alpha_k$ to~$p$,  and finally  along the full loop~$a$ back to~$p$.
 If $\varepsilon_p(a)=-1$, then the loop $b\circ_{p,q} a  \circ_{s,t} c$  starts at~$p$ and goes:   along $\alpha_k$   to~$q$, then along the full loop~$b$ back to~$q$,  then  along $\alpha_k$ to~$p$, then along the full  loop~$a$ back to~$p$,  then along $\alpha_k$ to~$t$, then along the full loop~$c$ back to~$t$,  and finally  along~$\alpha_k$ to~$p$.  
 
 In both cases, 
\begin{equation}\label{monoeg+}\varphi_\omega(k,p,q)=  \sum_{(l,s,t)\in T_\omega (a,c)}\,    \varepsilon(\omega, l)   
\,\varepsilon_s(a)  \,\varepsilon_t(c) \langle b\circ_{p,q} a  \circ_{s,t} c \rangle. \end{equation}
We call the summands  corresponding to    the triples $(l,s,t)\in  T_\omega (a,c)$ with $p\neq s$ the 4-tuple  terms. The   summands with $p=s$ (and $ k=l$) are  called  3-tuple terms.

Similarly, for $s\in b \cap \alpha_l$,  the homotopy class   $\langle (a\circ_{p,q}b)_s c_t \rangle$ is represented by the loop $a\circ_{p,q} b  \circ_{s,t} c$ obtained by grafting~$a$ and~$c$ to~$b$  via a path  in $\alpha_k$ from $p\in a\cap \alpha_k$ to $ q\in b\cap \alpha_k$ and  a path  in $\alpha_l$  from $s\in b\cap \alpha_l$ to $t\in c\cap \alpha_l$. A precise description  of this loop  also includes two cases determined by  whether or not $q=s$; we leave the details  to the reader. Thus,   
\begin{equation}\label{monoeg++}\psi_\omega (k,p,q)=  \sum_{(l,s,t) \in T_\omega ( b,c) }\,    \varepsilon(\omega, l)   
\,\varepsilon_s(b)  \,\varepsilon_t(c) \langle a\circ_{p,q} b  \circ_{s,t} c \rangle. \end{equation} We call the summands  corresponding to    the triples
   $(l,s,t)\in  T_\omega (b,c)$ with $q\neq s$ the 4-tuple  terms. The summands  with $q=s$ (and   $ k=l$)   are  called   3-tuple terms.
   
   Substituting these expressions  for  $\varphi, \psi$ in \eqref{monoeg}, we expand  $(x \bullet_{\omega} y )\bullet_\omega z$ as a linear combination of  4-tuple  and  3-tuple terms. Then Formula~\eqref{idyvg3n} yields such an expansion of $u_\omega (x,y,z)$    and Formula \eqref{idvnvnyvg3n+}   yields such an expansion of  $P(x,y,z)$. The  (total) contribution   of  the     4-tuple terms to $  P (x,y,z)$ is denoted by $  P_4 (x,y,z)$, and the  (total) contribution   of  the     3-tuple terms to $  P (x,y,z)$ is denoted by $  P_3 (x,y,z)$.
   We stress that $  P (x,y,z) = P_4 (x,y,z)+  P_3(x,y,z)$.
   
   % Given a linear combination~$L$ of   4-tuple  and  3-tuple terms, we write $4\#L$ and $3\#L$ for   the total contributions to~$L$  of  the     4-tuple terms and 3-tuple terms, respectively,  so that $L=4\#L+3\#L$.

We  prove  next  that  $P_4 (x,y,z)=P_4 (y,x,z)$.  Note first that   each point  $p,q,s,t$   in a 4-tuple (or a 3-tuple)  term    is traversed by exactly  one of the loops $a,b,c$. We will write $\varepsilon_p, \varepsilon_q, \varepsilon_s, \varepsilon_t$ for the corresponding signs $\pm 1$. For example, $\varepsilon_p=\varepsilon_p (a), \varepsilon_q=\varepsilon_q(b)$, etc. Also   set
  $\varepsilon(\omega, k, l) = \varepsilon(\omega, k)  \, \varepsilon(\omega, l) $.   In this notation, the  contribution  of  4-tuple terms to $(x\bullet_{\omega} y) \bullet_{\omega} z$
is equal to $\varphi_\omega^{x,y,z} +\psi_\omega^{x,y,z}$
where
$$\varphi_\omega^{x,y,z}= \sum_{ \begin{array}[b]{r}
     { (k,p,q) \in T_\omega (a,b) }\\
      {(l,s,t)\in T_\omega (a,c), p\neq s }
    \end{array} } \,  \varepsilon(\omega, k, l) \,\varepsilon_p \varepsilon_q \varepsilon_s \varepsilon_t  \langle b\circ_{p,q} a  \circ_{s,t} c \rangle$$
  and $$\psi_\omega^{x,y,z}= \sum_{ \begin{array}[b]{r}
     { (k,p,q) \in T_\omega (a,b) }\\
      {(l,s,t)\in T_\omega (b,c), q\neq s }
    \end{array} } \,   \varepsilon(\omega, k, l)  \,\varepsilon_p \varepsilon_q \varepsilon_s \varepsilon_t \langle a\circ_{p,q} b  \circ_{s,t} c \rangle. $$
  To     compute $\varphi_\omega^{y,z,x}$ and $\psi_\omega^{y,z,x}$, we cyclically permute $x,y,z$ and $a,b,c$ in the formulas above via  $a\mapsto b \mapsto c \mapsto a$.  It is convenient to  simultaneously permute the indices~$k, l$ and   permute the labels $p,q,s,t$ via    $p\mapsto s \mapsto q \mapsto t \mapsto p$. This gives 
  $$\varphi_\omega^{y,z,x}= \sum_{ \begin{array}[b]{r}
     { (l,s,t) \in  T_\omega (b, c) }\\
      {(k,q,p)\in T_\omega (b,a), s\neq q }
    \end{array} }\,   \varepsilon(\omega, k, l)  \,\varepsilon_p \varepsilon_q \varepsilon_s \varepsilon_t \,  \langle c\circ_{s,t} b  \circ_{q,p} a \rangle$$
  and $$\psi_\omega^{ y,z,x}= \sum_{ \begin{array}[b]{r}
     { (l,s,t) \in T_\omega (b,c) }\\
      {(k,q,p)\in T_\omega (c, a) , t\neq q}
    \end{array} }\,   \varepsilon(\omega, k, l)  \,\varepsilon_p \varepsilon_q \varepsilon_s \varepsilon_t \, \langle b\circ_{s,t} c \circ_{q,p} a \rangle. $$
 Applying the same permutations again, we get 
 $$\varphi_\omega^{z,x,y}= \sum_{ \begin{array}[b]{r}
     { (k,q,p) \in T_\omega (c,a) }\\
      {(l,t,s)\in T_\omega (c,b) , q\neq t}
    \end{array} }\,   \varepsilon(\omega, k, l)  \,\varepsilon_p \varepsilon_q \varepsilon_s \varepsilon_t\,  \langle a\circ_{q,p } c \circ_{t,s} b \rangle .$$
To compute $\psi_\omega^{ z,x ,y}$,  we  apply to $\psi_\omega^{ y,z,x}$  the   permutation  $a\mapsto b \mapsto c \mapsto a$  and the following permutation of the   indices:   $p\mapsto q\mapsto p, s \mapsto t \mapsto s$, $k\mapsto k, l \mapsto l$. Thus, 
  $$\psi_\omega^{ z,x ,y}= \sum_{ \begin{array}[b]{r}
     { (l,t,s) \in T_\omega (c,a) }\\
      {(k,p,q)\in T_\omega (a,b) , s\neq p }
    \end{array} }\,   \varepsilon(\omega, k, l)  \,\varepsilon_p \varepsilon_q \varepsilon_s \varepsilon_t\, \langle c\circ_{t,s} a \circ_{p,q} b \rangle. $$
    We conclude that the contribution  of  4-tuple terms to 
   $u_\omega (x,y,z)$ is equal to $$\varphi_\omega^{x,y,z}+\psi_\omega^{x,y,z}+ \varphi_\omega^{y,z,x} + \psi_\omega^{y,z,x}+ \varphi_\omega^{z,x,y}  + \psi_\omega^{z,x,y}
  = \Delta_\omega^{x,y,z}+ \Delta_\omega^{y,z,x}+ \Delta_\omega^{z,x,y}$$
for 
$$ \Delta_\omega^{x,y,z}=\psi_\omega^{x,y,z}+\varphi_\omega^{y,z,x}, \,\,\,\Delta_\omega^{y,z,x}=\psi_\omega^{y,z,x}+\varphi_\omega^{z,x,y}, \,\,\,  \Delta_\omega^{z,x,y}=\psi_\omega^{z,x,y}+\varphi_\omega^{x,y,z}.$$
Then 
  $$P_4(x,y,z)= ( \Delta_\omega^{x,y,z} + \Delta_{\overline \omega}^{x,y,z}) + ( \Delta_\omega^{y,z,x} + \Delta_{\overline \omega}^{y,z,x})+ ( \Delta_\omega^{z,x,y} + \Delta_{\overline \omega}^{z,x,y}).$$ We now compute all the $\Delta$'s. Comparing the expansions of  $\psi_\omega^{x,y,z} $ and $\varphi_\omega^{y,z,x}$ above, we observe that their summands are defined  by the same formula; here we   use the obvious fact  that the loops $ a\circ_{p,q} b  \circ_{s,t} c $ and  $c\circ_{s,t} b  \circ_{q,p} a$  are freely homotopic provided $q\neq s$.  The summation in these two expansions goes over complementary sets of indices as the inclusion    $(k,q,p)\in T_\omega (b,a)$   holds if and only if $ (k,p,q) \in T(a,b) \setminus T_\omega (a,b) $.  
  Therefore 
  $$ \Delta_\omega^{x,y,z}=  \sum_{ \begin{array}[b]{r}
     { (k,p,q) \in T (a,b) }\\
      {(l,s,t)\in T_\omega (b,c), s\neq q }
    \end{array} }\,   \varepsilon(\omega, k, l)  \,\varepsilon_p \varepsilon_q \varepsilon_s \varepsilon_t\, \langle a\circ_{p,q} b  \circ_{s,t} c \rangle. $$
  Analogously, 
  $$\Delta_\omega^{y,z,x} =\sum_{ \begin{array}[b]{r}
     { (l,s,t) \in T (b,c) }\\
      {(k,q,p)\in T_\omega (c, a), q\neq t }
    \end{array} }\,   \varepsilon(\omega, k, l)  \,\varepsilon_p \varepsilon_q \varepsilon_s \varepsilon_t\,  \langle b\circ_{s,t} c \circ_{q,p} a \rangle $$
  and
  $$\Delta_\omega^{z,x,y} =\sum_{ \begin{array}[b]{r}
     { (k,p,q) \in T_\omega (a,b) }\\
      {(l,s,t)\in T (a,c), s\neq p }
    \end{array} }\,   \varepsilon(\omega, k, l)  \,\varepsilon_p \varepsilon_q \varepsilon_s \varepsilon_t\,  \langle b\circ_{p,q} a  \circ_{s,t} c \rangle.$$
 Using   that $\varepsilon(\omega, k, l) = \varepsilon(\overline \omega, k, l) $, we  deduce from  these expressions   that
$$ \Delta_\omega^{x,y,z} + \Delta_{\overline \omega}^{x,y,z} = \sum_{ \begin{array}[b]{r}
     { (k,p,q) \in T (a,b) }\\
      {(l,s,t)\in T  (b,c), s\neq q }
    \end{array} }\,   \varepsilon(\omega, k, l)  \,\varepsilon_p \varepsilon_q \varepsilon_s \varepsilon_t\,  \langle a\circ_{p,q} b  \circ_{s,t} c \rangle, $$
  $$ \Delta_\omega^{y,z,x} + \Delta_{\overline \omega}^{y,z,x} =  \sum_{ \begin{array}[b]{r}
     { (l,s,t) \in T (b,c) }\\
      {(k,q,p)\in T (c, a), q\neq t }
    \end{array} }\,   \varepsilon(\omega, k, l)  \,\varepsilon_p \varepsilon_q \varepsilon_s \varepsilon_t\,  \langle b\circ_{s,t} c \circ_{q,p} a \rangle ,  $$
  $$ \Delta_\omega^{z,x,y} + \Delta_{\overline \omega}^{z,x,y} =\sum_{ \begin{array}[b]{r}
     { (k,p,q) \in T(a,b) }\\
      {(l,s,t)\in T (a,c), s\neq p }
    \end{array} }\,   \varepsilon(\omega, k, l)  \,\varepsilon_p \varepsilon_q \varepsilon_s \varepsilon_t\,  \langle b\circ_{p,q} a  \circ_{s,t} c \rangle.$$
As a consequence,  the permutation of $x,y$  keeps $ \Delta_\omega^{y,z,x} + \Delta_{\overline \omega}^{y,z,x}$ and  transforms $ \Delta_\omega^{x,y,z} + \Delta_{\overline \omega}^{x,y,z} $ and $ \Delta_\omega^{z,x,y} + \Delta_{\overline \omega}^{z,x,y}$   into each other.   Hence,     $P_4 (x,y,z)=P_4 (y,x,z)$
and 
so
$$  P(x,y,z)-P(y,x,z)=P_3(x,y,z)-P_3(y,x,z)$$
where $P_3(x,y,z)$ is the sum of 
 the 3-tuple terms   in the expansion of $P(x,y,z)$.
 Therefore to prove  \eqref{key} we need  to check that
 \begin{equation}\label{00-}  P_3(x,y,z)-P_3(y,x,z)=\mu(x,y,z)-\mu(y,x,z). \end{equation}

 Observe that each 3-tuple term   in the expansion of $P(x,y,z)$  is associated with an element $k=l$ of $\pi_0(\alpha)$ and   pairwise distinct points $p \in a\cap \alpha_k, q\in b\cap \alpha_k, t\in c\cap \alpha_k$.
 For each such triple of points, set
 \begin{equation}\label{zz} j(p,q,t)=\delta (\varepsilon_p, \varepsilon_q, \varepsilon_t)  \langle a_p  b_q c_t  \rangle+
  (\delta (\varepsilon_p, \varepsilon_q, \varepsilon_t) -\varepsilon_p  \varepsilon_q  \varepsilon_t)  \langle a_p  c_t  b_q  \rangle  \end{equation} and 
  \begin{equation}\label{zz+} j'(p,q,t)=\delta (\varepsilon_p, \varepsilon_q, \varepsilon_t)  \langle a_p  c_t  b_q \rangle+
  (\delta (\varepsilon_p, \varepsilon_q, \varepsilon_t) -\varepsilon_p  \varepsilon_q  \varepsilon_t)  \langle a_p   b_q c_t    \rangle . \end{equation} Here $\delta$ is the function of 3 signs defined in Section~\ref{section3loops+}.  The loop  $ a_p b_q    c_t  $ is the product of three loops $a_p, b_q, c_t$   formed via  connecting  their base points $p,q,t$   by arbitrary paths in $\alpha_k$. (In other words, we treat  $\alpha_k$ as a big base point  for these  loops.)  The loop  $ a_p  c_t b_q$ is defined similarly. Note that by
   the cyclic symmetry of     free homotopy classes of loops, we have  \begin{equation}\label{cyc} 
 \langle  a_p b_q c_t  \rangle=   \langle b_q  c_t a_p  \rangle=\langle c_t a_p b_q  \rangle \quad {\text {and}} \quad  \langle  a_p c_t b_q \rangle= 
     \langle b_q a_p c_t  \rangle=\langle  c_t b_q a_p \rangle   . \end{equation}

  Let $P_{3,k} (x,y,z)$ be the sum of 3-tuple terms associated with    $k\in \pi_0(\alpha)$. Clearly,   $$P_3(x,y,z)=\sum_{k\in \pi_0(\alpha)} P_{3,k} (x,y,z).$$
We prove below  that for all~$k$,
\begin{equation}\label{00}  P_{3,k}(x,y,z)   =\sum_{p,q,t} j(p,q,t)  . \end{equation}
In this  and  similar sums   $  p,q,t$ run respectively over the sets  $ a\cap \alpha_k, b\cap \alpha_k, c\cap \alpha_k$.   We first explain that this formula implies \eqref{00-}. Indeed, 
using the invariance of~$\delta$ under permutations   and  \eqref{cyc}, we deduce from \eqref{00} that
$$  P_{3,k}(y,x,z)   =\sum_{p,q,t}  
 j'(p,q,t) .$$  
 Therefore
 $$ P_{3,k} (x,y,z)- P_{3,k} (y,x,z) =\sum_{p,q,t} \big  ( j(p,q,t) - j'(p,q,t) \big ) $$
 $$=\sum_{p,q,t}
  \varepsilon_p\varepsilon_q \varepsilon_t  \langle a_p b_q  c_t   \rangle -
  \sum_{p,q,t}
  \varepsilon_p\varepsilon_q \varepsilon_t  \langle a_p   c_t b_q   \rangle $$
  $$= \mu^3_{k} (x,y,z)- \mu^3_{k} (x,z,y)= \mu^3_{k} (x,y,z)- \mu^3_{k} (y,x,z).$$
 Adding up these equalities over 
 all $k\in \pi_0(\alpha)$, we  get \eqref{00-}.
 
 It remains to prove~\eqref{00}.  
Fix $k\in \pi_0(\alpha)$.
Observe that Formulas \eqref{monoeg}--\eqref{monoeg++}    simplify for 3-tuple terms. Indeed,   $\varepsilon(\omega, k) \, \varepsilon(\omega,   l)=+1$ as $k=l$. Also, if $s=p$, then $\varepsilon_s \varepsilon_p=+1$; if   $s=q$, then $\varepsilon_s \varepsilon_q=+1$.      By the computations above, the  contribution  of  the 3-tuple terms (with fixed~$k$)  to $(x\bullet_{\omega} y) \bullet_{\omega} z$
is   equal to $\Phi_{\omega,k}^{x,y,z} +\Psi_{\omega,k}^{x,y,z}$
where
$$\Phi_{\omega,k}^{x,y,z}=  \sum_{  
     { q<_\omega p, \, t<_\omega p}
  } \,  \varepsilon_q \, \varepsilon_t    \, \langle b\circ_{p,q} a  \circ_{p,t} c \rangle, $$
  $$\Psi_{\omega,k}^{x,y,z}=    \sum_{   t<_\omega  q <_\omega p}
    \,  \varepsilon_p   \, \varepsilon_t   \,  \langle a\circ_{p,q} b  \circ_{q,t} c \rangle $$
   It is understood that the sum runs over     $ p\in a\cap \alpha_k, q\in b\cap \alpha_k, t\in c\cap \alpha_k$ satisfying the indicated inequalities.
    The description of the loop $  b\circ_{p,q} a  \circ_{p,t} c  $ above shows that it is freely homotopic to  $  a_p b_q c_t   $ if $\varepsilon_p =1$ and 
  to $a_pc_t b_q$ if $\varepsilon_p =-1$. Thus, 
    $$\langle b\circ_{p,q} a  \circ_{p,t} c \rangle=\delta( \varepsilon_p) \langle a_p  b_q   c_t  \rangle+ (1-\delta(\varepsilon_p )) \langle  a_p c_t  b_q \rangle   
    $$ and 
     \begin{equation}\label{1a} 
    \Phi_{\omega,k}^{x,y,z}=  \sum_{  
     { q<_\omega p, \, t<_\omega p}
  } \,  \varepsilon_q \, \varepsilon_t    \, \big (\delta( \varepsilon_p) \langle a_p  b_q   c_t  \rangle+ (1-\delta(\varepsilon_p )) \langle  a_p c_t  b_q \rangle   \big ). \end{equation}
    Similarly,
     \begin{equation}\label{2a} 
    \Psi_{\omega,k}^{x,y,z}=    \sum_{   t<_\omega  q <_\omega p}
    \,  \varepsilon_p   \, \varepsilon_t   \, \big ( \delta( \varepsilon_q)  \langle a_p c_t b_q  \rangle + 
   (1-\delta( \varepsilon_q)) \langle a_p  b_q c_t \rangle  \big ). \end{equation}
   %$$ \langle a\circ_{p,q} b  \circ_{q,t} c \rangle=\frac{1-\varepsilon_q }{2} \langle b_q c_t a_p  \rangle + 
    %\delta( \varepsilon_q) \langle b_q    a_p c_t \rangle.$$
Cyclically permuting  $(x,y,z)$,   $(a,b,c)$,      $(p,q,t)$, we get 
  \begin{equation}\label{3a} 
    \Phi_{\omega,k}^{y,z,x}=  \sum_{  
     { t<_\omega q, \, p<_\omega q}
  } \,  \varepsilon_t \, \varepsilon_p    \, \big (
    \delta( \varepsilon_q)  \langle a_p  b_q c_t \rangle + (1-\delta( \varepsilon_q)) \langle a_p c_t b_q  \rangle \big ) ,\end{equation}
     \begin{equation}\label{4a} 
    \Psi_{\omega,k}^{y,z,x}=    \sum_{  p<_\omega  t <_\omega q}
    \,  \varepsilon_q   \, \varepsilon_p   \, \big ( \delta( \varepsilon_t) \langle a_p c_t b_q  \rangle + 
   (1-\delta( \varepsilon_t))  \langle  a_p  b_q c_t  \rangle \big ),  \end{equation}
     \begin{equation}\label{5a} 
    \Phi_{\omega,k}^{z,x,y}=  \sum_{  
     {p<_\omega t, \,q<_\omega t}
  } \,  \varepsilon_p\, \varepsilon_q    \, \big ( 
    \delta( \varepsilon_t)  \langle a_p  b_q c_t  \rangle +(1-\delta( \varepsilon_t)) \langle a_p c_t b_q  \rangle \big ) , \end{equation}
     \begin{equation}\label{6a} 
    \Psi_{\omega,k}^{z,x,y}=    \sum_{  q<_\omega p<_\omega t}
    \,  \varepsilon_t   \, \varepsilon_q   \, \big ( \delta( \varepsilon_p) \langle a_p c_t b_q  \rangle + 
    (1-\delta (\varepsilon_p))   \langle a_p  b_q c_t  \rangle \big ). \end{equation}
    The  contribution of 3-tuple terms  (with given~$k$) to $u_\omega (x,y,z)$ is   the sum of 6 terms \eqref{1a}--\eqref{6a}. Then, by \eqref{idvnvnyvg3n+},    $P_{3,k}(x,y,z)$  is the sum of these 6 terms and 6 similar terms  obtained  by replacing~$\omega$ with~$\overline{\omega}$.  Under this replacement, the only change  on the right-hand sides of Formulas  \eqref{1a}--\eqref{6a} concerns  the summation domain. For example,  the summation domain in \eqref{1a}  changes from 
    the set of triples $p,q,t$ such that  ${ q<_\omega p, \, t<_\omega p}$ to the set of triples $p,q,t$ such that ${ q<_{\overline{\omega}} p, \, t<_{\overline{\omega}} p}$. The latter condition may be rewritten as  ${ p<_{ {\omega}} q, \, p<_{ {\omega}} t}$.  
    
  For $p\in a\cap \alpha_k, q \in b\cap \alpha_k , t \in c\cap \alpha_k$,     consider the 12 terms as above and pick their    $(p,q,t)$-summands (some of the  $(p,q,t)$-summands   may be  zero).    We claim that the sum of these  12 summands is equal to $j(p,q,t)$ for all $p,q,t$. 
    This clearly    implies  Formula~\eqref{00}. 
To prove our claim,   we consider possible   positions of the points $p,q,t$ on $\alpha_k$. Replacing, if necessary, $\omega$ by~$\overline \omega$, we can assume that $p<_\omega q$.  This leaves us with 3 cases:
(a)  $t<_\omega  p$; (b) $p<_\omega  t <_\omega  q$, and (c) $q <_\omega t$. In Case (a), only the $(p,q,t)$-summands of $\Phi^{y,z,x}_{\omega, k}$, $\Phi^{z,x,y}_{\overline \omega, k}$, and 
$\Psi^{z,x,y}_{\overline \omega, k}$ may be non-zero  and their sum is   
$$ \varepsilon_t \, \varepsilon_p   
    \delta( \varepsilon_q)  \langle a_p  b_q c_t \rangle + \varepsilon_t \, \varepsilon_p    (1-\delta( \varepsilon_q)) \langle a_p c_t b_q  \rangle $$
    $$+ \varepsilon_p\, \varepsilon_q    
    \delta( \varepsilon_t)  \langle a_p  b_q c_t  \rangle + \varepsilon_p\, \varepsilon_q  (1-\delta( \varepsilon_t)) \langle a_p c_t b_q  \rangle$$
    $$+\varepsilon_t   \, \varepsilon_q    \delta( \varepsilon_p) \langle a_p c_t b_q  \rangle + 
  \varepsilon_t   \, \varepsilon_q     (1-\delta (\varepsilon_p))   \langle a_p  b_q c_t  \rangle =j(p,q,t) $$  
where the last equality follows from  \eqref{dd} and \eqref{dd+}.
 In Case (b), only the $(p,q,t)$-summands of $\Phi^{y,z,x}_{\omega, k}$,  $\Psi^{y,z,x}_{\omega, k}$, and $\Phi^{x,y,z}_{\overline \omega, k}$  may be non-zero and their sum is equal to
$$   \varepsilon_t \, \varepsilon_p     
    \delta( \varepsilon_q)  \langle a_p  b_q c_t \rangle +  \varepsilon_t \, \varepsilon_p      (1-\delta( \varepsilon_q)) \langle a_p c_t b_q  \rangle$$
    $$+\varepsilon_q   \, \varepsilon_p   \delta( \varepsilon_t) \langle a_p c_t b_q  \rangle + 
  \varepsilon_q   \, \varepsilon_p  (1-\delta( \varepsilon_t))  \langle  a_p  b_q c_t  \rangle$$
  $$+ \varepsilon_q \, \varepsilon_t    \delta( \varepsilon_p) \langle a_p  b_q   c_t  \rangle+\varepsilon_q \, \varepsilon_t  (1-\delta(\varepsilon_p )) \langle  a_p c_t  b_q \rangle = j(p,q,t). $$  
  In Case (c), only the $(p,q,t)$-summands of $\Phi^{z,x,y}_{\omega, k}$,  $\Phi^{x,y,z}_{\overline \omega, k}$, and $\Psi^{x,y,z}_{\overline \omega, k}$ may be non-zero and their sum is equal to 
    $$   \varepsilon_t \, \varepsilon_p     
    \delta( \varepsilon_q)  \langle a_p  b_q c_t \rangle + \varepsilon_t \, \varepsilon_p   (1-\delta( \varepsilon_q)) \langle a_p c_t b_q  \rangle $$
    $$+  \varepsilon_q \, \varepsilon_t    \delta( \varepsilon_p) \langle a_p  b_q   c_t  \rangle+ \varepsilon_q \, \varepsilon_t   (1-\delta(\varepsilon_p )) \langle  a_p c_t  b_q \rangle  $$
  $$+   \varepsilon_p   \, \varepsilon_t     \delta( \varepsilon_q)  \langle a_p c_t b_q  \rangle + 
  \varepsilon_p   \, \varepsilon_t     (1-\delta( \varepsilon_q)) \langle a_p  b_q c_t \rangle = j(p,q,t)$$
  This proves the claim above  and completes the proof of the theorem.

\egm

    \end{document}